\documentclass{amsart}

\usepackage[english]{babel}

\usepackage{amstext,amsmath,amsthm,amssymb,mathrsfs}
\usepackage{float}

\usepackage{txfonts}
\usepackage{enumitem}
\usepackage[T1]{fontenc}

\newcommand{\lra}{\ensuremath{\Leftrightarrow}}

\newcommand{\longra}{\ensuremath{\longrightarrow}}

\newcommand{\C}{\varmathbb{C}}
\newcommand{\R}{\varmathbb{R}}
\newcommand{\N}{\varmathbb{N}}

\newcommand{\Z}{\varmathbb{Z}}

\newcommand{\E}{\varmathbb{E}}

\newcommand{\T}{\varmathbb{T}}

\newcommand{\Prob}{\varmathbb{P}}

\newcommand{\norm}[1]{||#1||}
\newcommand{\norml}[1]{\left|\left|#1\right|\right|}
\newcommand{\normB}[1]{\Big|\Big|#1\Big|\Big|}

\newcommand{\normm}[1]{\left|\left|\left|#1\right|\right|\right|}

\DeclareMathAlphabet{\mathpzc}{OT1}{pzc}{m}{it}

\DeclareMathOperator{\supp}{supp}

\theoremstyle{plain}
\newtheorem{thm}{Theorem}[section]
\newtheorem{lemma}[thm]{Lemma}
\newtheorem{prop}[thm]{Proposition}
\newtheorem{cor}[thm]{Corollary}

\newtheorem{question}[thm]{Question}

\theoremstyle{definition}

\newtheorem{ex}[thm]{Example}

\theoremstyle{remark}
\newtheorem{remark}[thm]{Remark}

\begin{document}

\title[Difference Norms for Vector-Valued Bessel Potential Spaces]{Difference Norms for Vector-Valued Bessel Potential Spaces with an Application to Pointwise Multipliers}

\author{Nick Lindemulder}
\email{N.Lindemulder@tudelft.nl}
\address{Delft Institute of Applied Mathematics, Delft University of Technology, P.O. Box
5031, 2600 GA Delft, The Netherlands}

\subjclass[2010]{Primary 46E40; Secondary 42B15, %42B25,
46B09, 46E30, 46E35}

\keywords{ Bessel potential space, difference norm, pointwise multiplier, UMD space, randomized Littlewood-Paley decomposition, Fourier multiplier, $\mathcal{R}$-boundedness, $A_{p}$-weight}
\date{\today}

\begin{abstract}
In this paper we prove a randomized difference norm characterization for Bessel potential spaces with values in UMD Banach spaces. The main ingredients are $\mathcal{R}$-boundedness results for Fourier multiplier operators, which are of independent interest. As an application we characterize the pointwise multiplier property of the indicator function of the half-space on these spaces. All results are proved in the setting of weighted spaces.
\end{abstract}

\maketitle

\section{Introduction}
Vector-valued Sobolev and Bessel potential spaces are important in the $L^{p}$-approach to abstract evolution and integral equations,
both in the deterministic setting (cf. e.g. \cite{Amann_Lin_and_quasilinear_parabolic,Pruss,Zacher}) and in the stochastic setting (cf. e.g. \cite{Desch&Londen_max-reg_stochastic_integral_equations,NVW_maximal-Lp-reg_stochastic_evolution_equations,NVW_Stochastic_maximal_Lp-reg}).
Here a central role is played by the Banach spaces that have the so-called UMD property (unconditionality of martingale differences);
see Section~\ref{subsec:sec:prereq:UMD&randomization} and the remarks below.
The class of Banach spaces that have UMD includes all Hilbert spaces, $L^{p}$-spaces with $p \in (1,\infty)$ and the reflexive Sobolev spaces, Triebel-Lizorkin spaces, Besov spaces and Orlicz spaces.

Let $X$ be a Banach space, $s \in \R$ and $p \in (1,\infty)$. The Bessel potential space $H^{s}_{p}(\R^{d};X)$ is defined in the usual Fourier analytic way via the Bessel potential operator $\mathcal{J}_{s} = (I-\Delta)^{s/2}$ based on the Lebesgue-Bochner space $L^{p}(\R^{d};X)$; see Section~\ref{subsec:sec:prereq:function_spaces}. If $X$ has UMD and $k \in \N$, then we have $H^{k}_{p}(\R^{d};X) = W^{k}(\R^{d};X)$, where $W^{k}_{p}(\R^{d};X)$ denotes the $k$-th order $X$-valued Sobolev space on $\R^{d}$ with integrability parameter $p$; see \cite{boek}, which also contains some converse results in this direction.
Furthermore, if $X$ has UMD and $s=k+\theta$ with $k \in \N$ and $\theta \in [0,1)$, then $H^{s}_{p}(\R^{d};X)$ can be realized as the complex interpolation space
\[
H^{s}_{p}(\R^{d};X) = [W^{k}_{p}(\R^{d};X),W^{k+1}_{p}(\R^{d};X)]_{\theta}.
\]

In the scalar-valued case $X = \C$,
Strichartz \cite{Strichartz_multipliers_fractional_Sobolev} characterized the Bessel potential space $H^{s}_{p}(\R^{d}) = H^{s}_{p}(\R^{d};\C)$,
with $s \in (0,1)$ and $p \in (1,\infty)$, by means of differences.
The characterization says that, for every $f \in L^{p}(\R^{d};\C)$, there is the equivalence of extended norms
\begin{equation}\label{eq:intro:difference_norm_scalar-case}
\norm{f}_{H^{s}_{p}(\R^{d};\C)} \eqsim \norm{f}_{L^{p}(\R^{d};\C)} +
\normB{ \Big(\int_{0}^{\infty}t^{-2s}\Big[ t^{-d}\int_{B(0,t)}\norm{\Delta_{h}f}_{\C}\,dh \Big]^{2} \frac{dt}{t} \Big)^{1/2} }_{L^{p}(\R^{d})},
\end{equation}
where $\Delta_{h}f = f(\,\cdot\,+h)-f$ for each $h \in \R^{d}$.
This extends to Hilbert spaces \cite[Section~6.1]{Walker_PhD-thesis}. In fact, given a Banach space $X$, the $X$-valued version of \eqref{eq:intro:difference_norm_scalar-case} is valid if and only if $X$ is isomorphic to a Hilbert space. Indeed, the $X$-valued version of  the right-hand side of \eqref{eq:intro:difference_norm_scalar-case} defines an extended norm on $L^{p}(\R^{d};X)$ which characterizes the Triebel-Lizorkin space $F^{s}_{p,2}(\R^{d};X)$ \cite[Section~2.3]{S&S_jena-notes}.
But the identity
\begin{equation}\label{eq:intro:classical_LP-decomp_H-space}
H^{s}_{p}(\R^{d};X) = F^{s}_{p,2}(\R^{d};X),
\end{equation}
i.e.\ the classical Littlewood-Paley decomposition for Bessel potential spaces, holds true if and only if $X$ is isomorphic to a Hilbert space \cite{Han&eyer_char_Hilbert_space,Rubio_de_Francia&Torrea_banach_techniques_vector-valued_Fourier_analysis}.
However, if $X$ is a Banach space with UMD, then one can replace \eqref{eq:intro:classical_LP-decomp_H-space} with a randomized Littlewood-Paley decomposition \cite{Meyries&Veraar_pt-multiplication} (see \eqref{eq:randomized-LP-decomp}), an idea which for the case $s=0$ originally goes back to Bourgain \cite{Bourgain_vector-valued_singular_integrals} and McConnell \cite{McConnell}. In \cite{Meyries&Veraar_pt-multiplication} this was used to investigate the pointwise multiplier property of the indicator function of the half-space on UMD-valued Bessel potential spaces. The randomized Littlewood-Paley decomposition will also play a crucial role in this paper to obtain a randomized difference norm characterization for UMD-valued Bessel potential spaces; see Theorem~\ref{thm:intro_main_result;difference}.

Since the early 1980's, randomization and martingale techniques have played a fundamental role in Banach space-valued analysis (cf. e.g. \cite{Burkholder_martingales_and_singular_integrals_Banach_spaces,Clement&Pagter&Suckochev&Witvliet_Schauder_decompositions, Denk&Hieber&Pruss_monograph2001,boek,boek2,Hytonen_survey-paper_vector-valued,Kalton&Weis_H_infty-calculus&square_functions, Kunstmann&Weis_lecture_notes,Rubio_de_Francia_survey_UMD, NVW_survey_stochastic_integration}). In particular, in Banach space-valued harmonic analysis and Banach space-valued stochastic analysis, a central role is played by the UMD spaces. Indeed, many classical Hilbert space-valued results from both areas have been extended to the UMD-valued case, and many of these extensions in fact characterize the UMD property. In vector-valued harmonic analysis, (one of) the first major breakthrough(s) is the deep result due to Bourgain \cite{Bourgain_Some_remarks_UMD} and Burkholder \cite{Burkholder_geometric_char_UMD} that a Banach space $X$ has UMD if and only if it is of class $\mathcal{HT}$, i.e.\ the Hilbert transform has a bounded extension to $L^{p}(\R;X)$ for some/all $p \in (1,\infty)$. As another major breakthrough we would like to mention the work of Weis \cite{Weis_operator-valued_FM_and_max-reg} on operator-valued Fourier multipliers on UMD-valued $L^{p}$-spaces ($p \in (1,\infty)$) with an application to the maximal $L^{p}$-regularity problem for abstract parabolic evolution equations. A central notion in this work is the $\mathcal{R}$-boundedness of a set of bounded linear operators on a Banach space, which is a randomized boundedness condition stronger than uniform boundedness; see Section~\ref{subsec:sec:prereq:UMD&randomization}. In Hilbert spaces it coincides with uniform boundedness and in $L^{p}$-spaces ($p \in [1,\infty)$), or more generally in Banach function spaces with finite cotype, it coincides with so-called $\ell^{2}$-boundedness. It follows from the work of Rubio de Francia (see \cite{Rubio_de_Francia1982,Rubio_de_Francia1983,Rubio_de_Francia1984} and \cite{Garcia-Cuarva&Rubio_de_Francia}) that $\ell^{2}$-boundedness in $L^{p}(\R^{d})$ ($p \in (1,\infty)$) is closely related to weighted norm inequalities; also see \cite{Gallarati&Lorist&Veraar}.

Randomization techniques also play an important role in this paper. As already mentioned above, we work with a randomized substitute of \eqref{eq:intro:classical_LP-decomp_H-space}. This approach naturally leads to the problem of determining the $\mathcal{R}$-boundedness of a sequence of Fourier multiplier operators. The latter forms a substantial part of this paper, which is also of independent interest; see Section~\ref{sec:R-bdd_FM}.

The results in this paper are proved in the setting of weighted spaces, which includes the unweighted case.
We consider weights from the so-called Muckenhoupt class $A_{p}$. This is a class of weights for which many harmonic analytic tools from the unweighted setting remain valid; see Section~\ref{subsec:prelim:weights}. An important example of an $A_{p}$-weight is the power weight $w_{\gamma}$, given by
\begin{equation}\label{eq:intro:power_weight}
w_{\gamma}(x_{1},x') = |x_{1}|^{\gamma}, \quad\quad (x_{1},x') \in \R^{d} = \R \times \R^{d-1},
\end{equation}
for the parameter $\gamma \in (-1,p-1)$. In the maximal $L^{p}$-regularity approach to parabolic evolution equations these power weights yield flexibility in the optimal regularity of the initial data (cf. e.g. \cite{Meyries&Schnaubelt_fractional_Sobolev_spaces_temporal_weights,Meyries&Schnaubelt_maximal_reg_temporal_weights,
Meyries&Veraar_traces&embeddings_anisotropic_function_spaces,Pruss&Simonett_maximal_reg_weighted_spaces}).

The following theorem is our main result. Before we can state it, we first need to explain some notation. We denote by $\{\varepsilon_{j}\}_{j \in \N}$ a Rademacher sequence on some probability space $(\Omega,\mathcal{F},\Prob)$, i.e. a sequence of independent symmetric $\{-1,1\}$-valued random variables on $(\Omega,\mathcal{F},\Prob)$. For a natural number $m \geq 1$ and a function $f$ on $\R^{d}$ with values in some vector space $X$, we write
\[
\Delta^{m}_{h}f(x) = \sum_{j=0}^{m}(-1)^{j}{m \choose j}f(x+(m-j)h), \quad\quad x \in \R^{d}, h \in \R^{d}.
\]

\begin{thm}\label{thm:intro_main_result;difference}
Let $X$ be a UMD Banach space, $s > 0$, $p \in (1,\infty)$, $w \in A_{p}(\R^{d})$ and $m \in \N$, $m>s$. Suppose that
\begin{itemize}
\item $K=1_{[-1,1]^{d}}$ in the unweighted case $w=1$; or
\item $K \in \mathcal{S}(\R^{d})$ is such that $\int_{\R}K(y)dy \neq 0$ in the general weighted case.
\end{itemize}
For all $f \in L^{p}(\R^{d},w;X)$ we then have the equivalence of extended norms
\begin{equation}\label{eq:thm:intro_main_result;difference}
\norm{f}_{H^{s}_{p}(\R^{d},w;X)}
\eqsim \norm{f}_{L^{p}(\R^{d},w;X)} + \sup_{J \in \N}\normB{\sum_{j=1}^{J}\varepsilon_{j}2^{js}\int_{\R^{d}}K(h)\Delta^{m}_{2^{-j}h}f\,dh}_{L^{p}(\Omega;L^{p}(\R^{d},w;X))}.
\end{equation}
\end{thm}

\begin{remark}
If $f \in H^{s}_{p}(\R^{d},w;X)$, then the finiteness of the supremum on the RHS of \eqref{eq:thm:intro_main_result;difference} actually implies the convergence of the sum $\sum_{j=1}^{\infty}\varepsilon_{j}2^{js}\int_{\R^{d}}K(h)\Delta^{m}_{2^{-j}h}f\,dh$ in $L^{p}(\Omega;L^{p}(\R^{d},w;X))$.
Moreover, \eqref{eq:thm:intro_main_result;difference} then takes the form
\[
\norm{f}_{H^{s}_{p}(\R^{d},w;X)}
\eqsim \norm{f}_{L^{p}(\R^{d},w;X)} + \normB{\sum_{j=1}^{\infty}\varepsilon_{j}2^{js}\int_{\R^{d}}K(h)\Delta^{m}_{2^{-j}h}f\,dh}_{L^{p}(\Omega;L^{p}(\R^{d},w;X))}.
\]
This follows from the convergence result \cite[Theorem~9.29]{Ledoux&Talagrand} together with the fact that $L^{p}(\R^{d},w;X)$ (as a UMD space) does not contain a copy $c_{0}$.
\end{remark}

\begin{remark}\label{rmk:thm:intro_main_result;difference}
We will in fact prove a slightly more general difference norm characterization for $H^{s}_{p}(\R^{d},w;X)$, namely Theorem~\ref{thm:difference_norms_separate_ineq_'sharp'}, where we consider kernels $K$ satisfying certain integrability conditions plus an $\mathcal{R}$-boundedness condition.
Here the $\mathcal{R}$-boundedness condition is only needed for the inequality '$\gtrsim$'. In the case $m=1$ it corresponds to the $\mathcal{R}$-boundedness of the convolution operators $\{ f \mapsto K_{t}*f : t=2^{j}, j \geq 1 \}$ in $\mathcal{B}(L^{p}(\R^{d},w;X))$, where $K_{t} = t^{d}K(t\,\cdot\,)$. For more information we refer to Section~\ref{subsec:sec:difference_norms;main_result}.
\end{remark}

To the best of our knowledge, Theorem~\ref{thm:intro_main_result;difference} is the first difference norm characterization for (non-Hilbertian) Banach space-valued Bessel potential spaces available in the literature. In the special case when $X$ is a UMD Banach function space, the norm equivalence from this theorem takes (with possibly different implicit constants), by the Khinthchine-Maurey theorem, the square function form
\[
\norm{f}_{H^{s}_{p}(\R^{d},w;X)}
\eqsim \norm{f}_{L^{p}(\R^{d},w;X)} + \normB{\Big(\sum_{j=1}^{\infty}\big|\,2^{js}\int_{\R^{d}}K(h)\Delta^{m}_{2^{-j}h}f\,dh\,\big|^{2}\Big)^{1/2}}_{L^{p}(\R^{d},w;X)};
\]
see Section~\ref{sec:BFS-case}.
In the unweighted scalar-valued case $X=\C$, this a discrete version for the case $q=2$ of the characterization \cite[Theorem~2.6.3]{Triebel_Theory_of_FS_II} of the Triebel-Lizorkin space $F^{s}_{p,q}(\R^{d})$ by weighted means of differences (recall \eqref{eq:intro:classical_LP-decomp_H-space}).
Furthermore, in the unweighted scalar-valued case $X=\C$, one can also think of it as a discrete analogue of Strichartz's characterization~ \eqref{eq:intro:difference_norm_scalar-case}.

As an application of Theorem~\ref{thm:intro_main_result;difference}, we characterize the boundedness of the indicator function $1_{\R^{d}_{+}}$ of the half-space $\R^{d}_{+}=\R_{+} \times \R^{d-1}$ as a pointwise multiplier on $H^{s}_{p}(\R^{d},w;X)$ in terms of a continuous inclusion of the corresponding scalar-valued Bessel potential space $H^{s}_{p}(\R^{d},w)$ into a certain weighted $L^{p}$-space; see Theorem~\ref{thm:pointwise_multiplier}.
The importance of the pointwise multiplier property of $1_{\R^{d}_{+}}$ lies in the fact that it served as one of the main ingredients of Seeley's result \cite{Seeley_interpolation_boundary} on the characterization of complex interpolation spaces of Sobolev spaces with boundary conditions.
As an application of an extension of Seeley's characterization to the weighted vector-valued case one could, for example, characterize the fractional power domains of the time derivative with zero initial conditions on $L^{p}(\R^{d}_{+},w_{\gamma};X)$.

\begin{thm}\label{thm:pointwise_multiplier}
Let $X \neq \{0\}$ be a UMD space, $s \in (0,1)$, $p \in (1,\infty)$ and $w \in A_{p}(\R^{d})$.
Let $w_{s,p}$ be the weight on $\R^{d} = \R \times \R^{d-1}$ given by $w_{s,p}(x_{1},x') := |x_{1}|^{-sp}w(x_{1},x')$ if $|x_{1}| \leq 1$ and $w_{s,p}(x_{1},x') := w(x_{1},x')$ if $|x_{1}| > 1$.
Then  $1_{\R^{d}_{+}}$ is a pointwise multiplier on $H^{s}_{p}(\R^{d},w;X)$ if and only if there is the inclusion
\begin{equation}\label{eq:thm:pointwise_multiplier;inclusion}
H^{s}_{p}(\R^{d},w) \hookrightarrow L^{p}(\R^{d},w_{s,p}).
\end{equation}
\end{thm}
In Section~\ref{subsec:closer_look_inclusion} we will take a closer look at the inclusion \eqref{eq:thm:pointwise_multiplier;inclusion}.
Based on embedding results from \cite{Meyries&Veraar_char_class_embeddings}, we will give explicit conditions (in terms of the weight and the parameters) for which this inclusion holds true. The important class of power weights \eqref{eq:intro:power_weight} is considered in Example~\ref{ex:voorbeelden_inclusie}.

In the situation of the above theorem, let $\bar{w}_{s,p}$ be the weight on $\R \times \R^{d-1}$ defined by $\bar{w}_{s,p}(x_{1},x') := |x_{1}|^{-sp}w(x_{1},x')$. Note that, in view of the inclusion $H^{s}_{p}(\R^{d},w) \hookrightarrow L^{p}(\R^{d},w)$, the inclusion \eqref{eq:thm:pointwise_multiplier;inclusion} is equivalent to the inclusion
\[
H^{s}_{p}(\R^{d},w) \hookrightarrow L^{p}(\R^{d},\bar{w}_{s,p}).
\]
%When thought of as an inequality, in the unweighted case the last inclusion is Hardy's inequality (with respect to the hyperplane $\{0\} \times \R^{d-1}$) for Bessel potential spaces (cf. \cite[Section~I.5]{Triebel_Structure_of_functions}).
In the unweighted scalar-valued case, the above theorem thus corresponds to a result of Triebel \cite[Section~2.8.6]{Triebel_Theory_of_FS_I} with $q=2$, which states that the multiplier property for $F^{s}_{p,q}(\R^{d})$ (recall \eqref{eq:intro:classical_LP-decomp_H-space}) is equivalent to the inequality
\[
\norm{x \mapsto |x_{1}|^{s}f(x)}_{L^{p}(\R^{d})} \lesssim \norm{f}_{F^{s}_{p,q}(\R^{d})},\quad\quad f \in F^{s}_{p,q}(\R^{d}).
\]
Similarly to Strichartz \cite{Strichartz_multipliers_fractional_Sobolev}, who used \eqref{eq:intro:difference_norm_scalar-case} to prove that $1_{\R^{d}_{+}}$ acts a pointwise multiplier on $H^{s}_{p}(\R^{d})$ in the parameter range
\[
-\frac{1}{p'} < s < \frac{1}{p}, \quad\quad \mbox{where}\quad \frac{1}{p}+\frac{1}{p'}=1,\footnote{This result is originally due to Shamir \cite{Shamir_pointwise_multiplier}. However, Strichartz \cite{Strichartz_multipliers_fractional_Sobolev} in fact obtained this result as a corollary to a more general pointwise multiplication result (in combination with a Fubini type theorem for Bessel potential spaces).}
\]
Triebel used a difference norm characterization in his proof.
Our proof is closely related to the proof of Triebel \cite[Section~2.8.6]{Triebel_Theory_of_FS_I}.

An alternative approach to pointwise multiplication is via the paraproduct technique (cf.\ e.g.\ the monograph of Runst and Sickel \cite{Runst&Sickel_monograph} for the unweighted scalar-valued setting).
Based on a randomized Littlewood-Paley decomposition, Meyries and Veraar \cite{Meyries&Veraar_pt-multiplication} followed such an approach to extend the classical result of Shamir \cite{Shamir_pointwise_multiplier} and Strichartz \cite{Strichartz_multipliers_fractional_Sobolev} to the weighted vector-valued case.
They in fact proved a more general pointwise multiplication result
for the important class of power weights $w_{\gamma}$ \eqref{eq:intro:power_weight}, $\gamma \in (-1,p-1)$,  in the UMD setting, from which the case of the characteristic function $1_{\R^{d}_{+}}$ can be derived.
Their main result \cite[Theorem~1.1]{Meyries&Veraar_pt-multiplication} says that, given a UMD Banach space $X$, $p \in (1,\infty)$ and $\gamma \in (-1,p-1)$, $1_{\R^{d}_{+}}$ is a pointwise multiplier on $H^{s}_{p}(\R^{d},w_{\gamma};X)$ in the parameter range
\[
-\frac{1+\gamma'}{p'} < s < \frac{1+\gamma}{p}, \quad\quad \mbox{where}\quad \frac{1}{p}+\frac{1}{p'}=1,\,\gamma' = -\frac{\gamma}{p-1}.
\]
For positive smoothness $s \geq 0$ this pointwise multiplication result is contained in Example~\ref{ex:voorbeelden_inclusie}, from which the case of negative smoothness $s \leq 0$ can be derived via duality.

The paper is organized as follows. Section~\ref{sec:prerequisites} is devoted to the necessary preliminaries. In Section~\ref{sec:R-bdd_FM} we treat $\mathcal{R}$-boundedness results for Fourier multiplier operators on $L^{p}(\R^{d},w;X)$. The results from this section form (together with a randomized Littlewood-Paley decomposition) the main tools for this paper, but are also of independent interest. In Section~\ref{sec:difference_norms} we state and prove the main result of this paper, Theorem~\ref{thm:difference_norms_separate_ineq_'sharp'}, from which Theorem~\ref{thm:intro_main_result;difference} can be obtained as a consequence. Finally, in Section~\ref{sec:pointwise-multiplier} we use difference norms to prove the pointwise multiplier Theorem~\ref{thm:pointwise_multiplier}, and we also take a closer look at the inclusion \eqref{eq:thm:pointwise_multiplier;inclusion} from this theorem.

\textbf{Notations and conventions.} All vector spaces are over the field of complex scalars $\C$.
$|A|$ denotes the Lebesgue measure of Borel set $A \subset \R^{d}$.
Given a measure space $(X,\mathscr{A},\mu)$, for $A \in \mathscr{A}$ with $\mu(A) \in (0,\infty)$ we write
\[
\fint_{A}d\mu = \frac{1}{\mu(A)}\int_{A}d\mu.
\]
For a function $f:\R^{d} \longra X$, with $X$ some vector space, we write $\tilde{f}(x)=f(-x)$ and, unless otherwise stated, $f_{t}(x) = t^{d}f(tx)$ for every $x \in \R^{d}$ and $t>0$. Given a Banach space $X$, we denote by $L^{0}(\R^{d};X)$ the space of equivalence classes of Lebesgue strongly measurable $X$-valued functions on $\R^{d}$. For $x \in \R^{d}$ and $r>0$ we write $Q[x,r] = x + [-r,r]^{d}$ for the cube centered at $x$ with side length $2r$.

\section{Prerequisites}\label{sec:prerequisites}

\subsection{UMD Spaces and Randomization}\label{subsec:sec:prereq:UMD&randomization}

The general references for this subsection are \cite{boek,boek2,Kunstmann&Weis_lecture_notes}.

A Banach space $X$ is called a UMD space if for any probability space $(\Omega,\mathcal{F},\Prob)$ and $p \in (1,\infty)$ it holds true that martingale differences are unconditional in $L^{p}(\Omega;X)$ (see \cite{Burkholder_martingales_and_singular_integrals_Banach_spaces,Rubio_de_Francia_survey_UMD} for a survey on the subject).
It is a deep result due to Bourgain and Burkholder that a Banach space $X$ has UMD if and only if it is of class $\mathcal{HT}$, i.e.\ the Hilbert transform has a bounded extension to $L^{p}(\R;X)$ for any/some $p \in (1,\infty)$. Examples of Banach spaces with the UMD property include all Hilbert spaces and all $L^{q}$-spaces with $q \in (1,\infty)$.

Throughout this paper, we fix a \emph{Rademacher sequence} $\{\varepsilon_{j}\}_{j \in \Z}$ on some probability space $(\Omega,\mathcal{F},\Prob)$, i.e. a sequence of independent symmetric $\{-1,1\}$-valued random variables on $(\Omega,\mathcal{F},\Prob)$. If necessary, we denote by $\{\varepsilon'_{j}\}_{j \in \Z}$ a second Rademacher sequence on some probability space $(\Omega',\mathcal{F}',\Prob')$ which is independent of the first.

Let $X$ be a Banach function space with finite cotype and let $p \in [1,\infty)$.\footnote{A Banach space $X$ has cotype $q \in [2,\infty]$ if $\left(\sum_{j=0}^{n}\norm{x_{j}}^{q}\right)^{1/q} \lesssim \norm{\sum_{j=0}^{n}\varepsilon_{j}x_{j}}_{L^{2}(\Omega;X)}$ for all $x_{0},\ldots,x_{n} \in X$.  We say that $X$ has finite cotype if it has cotype $q \in [2,\infty)$. The cotype of $L^{p}$ is the maximum of $2$ and $p$. Every UMD space has finite cotype.}
The Khinthchine-Maurey theorem says that, for all $x_{0},\ldots,x_{n} \in X$,
\begin{equation}\label{eq:Khintchine_Maurey}
\normB{\Big(\sum_{j=0}^{n}|x_{j}|^{2} \Big)^{1/2}}_{X} \eqsim \normB{\sum_{j=0}^{n}\varepsilon_{j}x_{j}}_{L^{p}(\Omega;X)}.
\end{equation}
In the special case $E=L^{q}(S)$ ($q \in [1,\infty)$) this easily follows from a combination of Fubini and the Kahane-Khintchine inequality. Morally, \eqref{eq:Khintchine_Maurey} means that square function estimates are equivalent to estimates for Rademacher sums.

The classical Littlewood-Paley inequality gives a two-sided estimate for the $L^{p}$-norm of a scalar-valued function by the $L^{p}$-norm of the square function corresponding to its dyadic spectral decomposition.
This classical inequality has a UMD Banach space-valued version, due to Bourgain \cite{Bourgain_vector-valued_singular_integrals} and McConnell \cite{McConnell}, in which the square function is replaced by a Rademacher sum (as in \eqref{eq:Khintchine_Maurey}; see the survey paper \cite{Hytonen_survey-paper_vector-valued}).
One of the main ingredients of this paper is a similar inequality for Bessel potential spaces, namely the randomized Littlewood-Paley decomposition \eqref{eq:randomized-LP-decomp}.

Let $X$ be a Banach space and $p \in [1,\infty]$. As a special case of the (Kahane) contraction principle, for all $x_{0},\ldots,x_{n} \in X$ and $a_{0},\ldots,a_{n} \in \C$ it holds that
\begin{equation}\label{eq:contraction_principle}
\normB{\sum_{j=0}^{n}a_{j}\varepsilon_{j}x_{j}}_{L^{p}(\Omega;X)} \leq 2|a|_{\infty}\normB{\sum_{j=0}^{n}\varepsilon_{j}x_{j}}_{L^{p}(\Omega;X)}.
\end{equation}

A family of operators $\mathcal{T} \subset \mathcal{B}(X)$ on a Banach space $X$ is called \emph{$\mathcal{R}$-bounded} if there exists a constant $C \geq 0$ such that for all $T_{0},\ldots,T_{N} \in \mathcal{T}$ and $x_{0},\ldots,x_{N} \in X$ it holds that
\begin{equation}\label{eq:R-bdd_def}
\normB{\sum_{j=0}^{N}\varepsilon_{j}T_{j}x_{j}}_{L^{2}(\Omega;X)} \leq C\,\normB{\sum_{j=0}^{N}\varepsilon_{j}x_{j}}_{L^{2}(\Omega;X)}.
\end{equation}
The moments of order $2$ above may be replaced by moments of any order $p$. The resulting least admissible constant is denoted by $\mathcal{R}_{p}(\mathcal{T})$. In the definition of $\mathcal{R}$-boundedness it actually suffices to check \eqref{eq:R-bdd_def} for distinct operators $T_{0},\ldots,T_{N} \in \mathcal{T}$.

A Banach space $X$ is said to have \emph{Pisier's contraction property} or \emph{property $(\alpha)$} if the contraction principle holds true for double  Rademacher sums (for some extra fixed multiplicative constant); see \cite[Definition~4.9]{Kunstmann&Weis_lecture_notes} for the precise definition. Every space $L^{p}$ with $p \in [1,\infty)$ enjoys property $(\alpha)$. Further examples are UMD Banach function spaces. However, the Schatten von Neumann class $\mathfrak{S}_{p}$ enjoys property $(\alpha)$ if and only if $p=2$.

A Banach space $X$ is said to have the \emph{triangular contraction property} or \emph{property ($\Delta$)} if there exists a constant $C \geq 0$ such that for all $\{x_{i,j}\}_{i,j=0}^{n} \subset X$
\[
\normB{ \sum_{0 \leq j \leq i \leq n}\varepsilon_{i}\varepsilon'_{j}x_{i,j} }_{L^{2}(\Omega \times \Omega';X)}
\leq C\,\normB{ \sum_{i,j=0}^{n}\varepsilon_{i}\varepsilon'_{j}x_{i,j} }_{L^{2}(\Omega \times \Omega';X)};
\]
see \cite{Kalton&Weis_sums_of_closed_operators}.
The moments of order $2$ above may be replaced by moments of any order $p$. The resulting least admissible constant is denoted by $\Delta_{p,X}$. Every space with Pisier's contraction property trivially has the triangular contraction property. For vector-valued $L^{p}$-spaces we have $\Delta_{p,L^{p}(S;X)} = \Delta_{p,X}$. Furthermore, every UMD space has the triangular contraction property.

Let $X$ be a Banach space.
The space $\mathrm{Rad}(X)$ is the linear space consisting of all sequences $\{x_{j}\}_{j} \subset X$ for which $\sum_{j \in \N}\varepsilon_{j}x_{j}$ defines a convergent series in $L^{2}(\Omega;X)$. It becomes a Banach space under the norm $\norm{\{x_{j}\}_{j}}_{\mathrm{Rad}(X)} := \norm{\sum_{j \in \N}\varepsilon_{j}x_{j}}_{L^{2}(\Omega;X)}$; see \cite{boek2,Kaip&Saal,Kunstmann&Weis_lecture_notes}.

\subsection{Muckenhoupt Weights}\label{subsec:prelim:weights}

In this subsection the general reference is \cite{Grafakos_modern}.

A \emph{weight} is a positive measurable function on $\R^{d}$ that takes it values almost everywhere in $(0,\infty)$.
Let $w$ be a weight on $\R^{d}$. We write $w(A) = \int_{A}w(x)\,dx$ when $A$ is Borel measurable set in $\R^{d}$.
Furthermore, given a Banach space $X$ and $p \in [1,\infty)$, we define the weighted Lebesgue-Bochner space $L^{p}(\R^{d},w;X)$ as the Banach space of all $f \in L^{0}(\R^{d};X)$ for which
\[
\norm{f}_{L^{p}(\R^{d},w;X)} := \left( \int_{\R^{d}}\norm{f(x)}_{X}^{p}\,w(x)\,dx \right)^{1/p} < \infty.
\]

For $p \in [1,\infty]$ we denote by $A_{p}=A_{p}(\R^{d})$ the class of all Muckenhaupt $A_{p}$-weights, which are all the locally integrable weights for which the $A_{p}$-characteristic $[w]_{A_{p}} \in [1,\infty]$ is finite; see \cite[Chapter~9]{Grafakos_modern} for more details.
Let us recall the following facts:
\begin{itemize}
\item $A_{\infty} = \bigcup_{p \in (1,\infty)}A_{p}$, which often also taken as definition;
\item For $p \in (1,\infty)$ and a weight $w$ on $\R^{d}$: $w \in A_{p}$ if and only if $w^{-\frac{1}{p-1}} \in A_{p'}$, where $\frac{1}{p}+\frac{1}{p'}=1$;
\item For a weight $w$ on $\R^{d}$ and $\lambda>0$: $[w(\lambda\,\cdot\,)]_{A_{p}} = [w]_{A_{p}}$;
\item For $p \in [1,\infty)$ and $w \in A_{\infty}(\R^{d})$: $\mathcal{S}(\R^{d}) \stackrel{d}{\hookrightarrow} L^{p}(\R^{d},w)$;
\item The Hardy-Littlewood maximal operator $M$ is bounded on $L^{p}(\R^{d},w)$ if (and only if) $w \in A_{p}$.
\end{itemize}

An example of an $A_{\infty}$-weight is the power weight $w_{\gamma}$ \eqref{eq:intro:power_weight} for $\gamma > -1$. Given $p \in (1,\infty)$, we have $w_{\gamma} \in A_{p}$ if and only if $\gamma \in (-1,p-1)$. Also see \eqref{eq:example_power_weight} for a slight variation.

A function $f:\R^{d} \longra \R$ is called \emph{radially decreasing} if it is of the form $f(x)=g(|x|)$ for some decreasing function $g:\R \longra \R$.
We define $\mathscr{K}(\R^{d})$ as the space of all $k \in L^{1}(\R^{d})$ having a radially decreasing integrable majorant, i.e., all $k \in L^{1}(\R^{d})$ for which there exists a radially decreasing $\psi \in L^{1}(\R^{d})^{+}$ with $|k| \leq \psi$. Equipped with the norm
\[
\norm{k}_{\mathscr{K}(\R^{d})} := \inf\left\{ \norm{\psi}_{L^{1}(\R^{d})} : \psi \in L^{1}(\R^{d})^{+} \:\mbox{radially decreasing}, |k| \leq \psi \right\},
\]
$\mathscr{K}(\R^{d})$ becomes a Banach space. Note that, given $k \in \mathscr{K}(\R^{d})$ and $t > 0$, we have $k_{t} = t^{d}k(t\,\cdot\,) \in \mathscr{K}(\R^{d})$ with $\norm{k_{t}}_{\mathscr{K}(\R^{d})} = \norm{k}_{\mathscr{K}(\R^{d})}$.

Let $X$ be a Banach space. For $k \in \mathscr{K}(\R^{d})$ we have the pointwise estimate
\begin{equation*}\label{eq:estimate_conv_radicall_decr}
\int_{\R^{d}}|k(x-y)|\,\norm{f(y)}_{X}\,dy \leq \norm{k}_{\mathscr{K}(\R^{d})}M(\norm{f}_{X})(x),\quad\quad f \in L^{1}_{loc}(\R^{d};X), x \in \R^{d}.
\end{equation*}
As a consequence, if $p \in (1,\infty)$ and $w \in A_{p}(\R^{d})$, then $k$ gives rise to a well-defined bounded convolution operator $k*: f \mapsto k*f$ on $L^{p}(\R^{d},w;X)$, given by the formula
\[
k*f(x) = \int_{\R^{d}}k(x-y) f(y)\,dy, \quad\quad x \in \R^{d},
\]
for which we have the norm estimate $\norm{k*}_{\mathcal{B}(L^{p}(\R^{d},w;X))} \lesssim_{p,d,w} \norm{k}_{\mathscr{K}(\R^{d})}$.

\subsection{Function Spaces}\label{subsec:sec:prereq:function_spaces}

As general reference to the theory of vector-valued distributions we mention \cite{Amann_distributions} (and \cite[Section~III.4]{Amann_Lin_and_quasilinear_parabolic}). For vector-valued function spaces we refer to \cite{boek,S&S_jena-notes} (unweighted setting) and \cite{Meyries&Veraar_pt-multiplication} (weighted setting) and the references given therein.

Let $X$ be a Banach space.
The space of $X$-valued tempered distributions $\mathcal{S}'(\R^{d};X)$ is defined as $\mathcal{S}'(\R^{d};X) := \mathcal{L}(\mathcal{S}(\R^{d}),X)$, the space of continuous linear operators from $\mathcal{S}(\R^{d})$ to $X$, equipped with the locally convex topology of bounded convergence.
Standard operators (derivative operators, Fourier transform, convolution, etc.) on $\mathcal{S}'(\R^{d};X)$ can be defined as in the scalar-case, cf. \cite[Section~III.4]{Amann_Lin_and_quasilinear_parabolic}.

Let $p \in (1,\infty)$ and $w \in A_{p}(\R^{d})$. Then $w^{1-p'} = w^{-\frac{1}{p-1}} \in A_{p'}$, so that $\mathcal{S}(\R^{d}) \stackrel{d}{\hookrightarrow} L^{p'}(\R^{d},w^{1-p'})$. By H\"older's inequality we find that $L^{p}(\R^{d},w;X) \hookrightarrow \mathcal{S}'(\R^{d};X)$ in the natural way.
For each $s \in \R$ we can thus define the Bessel potential space $H^{s}_{p}(\R^{d},w;X)$ as the space of all $f \in \mathcal{S}'(\R^{d};X)$
for which $\mathcal{J}_{s}f \in L^{p}(\R^{d},w;X)$, equipped with the norm $\norm{f}_{H^{s}_{p}(\R^{d},w;X)} := \norm{\mathcal{J}_{s}f}_{L^{p}(\R^{d},w;X)}$; here $\mathcal{J}_{s} \in \mathcal{L}(\mathcal{S}'(\R^{d};X))$ is the Bessel potential operator given by
\[
\mathcal{J}_{s}f := \mathscr{F}^{-1}[(1+|\,\cdot\,|^{2})^{s/2}\hat{f}], \quad\quad f \in \mathcal{S}'(\R^{d};X).
\]
Furthermore, for each $n \in \N$ we can define the Sobolev space $W^{n}_{p}(\R^{d},w;X)$ as the space of all $f \in \mathcal{S}'(\R^{d};X)$ for which $\partial^{\alpha}f \in L^{p}(\R^{d},w;X)$ for every $|\alpha| \leq n$, equipped with the norm $\norm{f}_{W^{n}_{p}(\R^{d},w;X)}:=\sum_{|\alpha| \leq n}\norm{\partial^{\alpha}f}_{L^{p}(\R^{d},w;X)}$.
Note that $H^{0}_{p}(\R^{d},w;X) = L^{p}(\R^{d},w;X) = W^{0}_{p}(\R^{d},w;X)$. If $X$ is a UMD space, then we have $H^{n}_{p}(\R^{d},w;X) = W^{n}_{p}(\R^{d},w;X)$. In the reverse direction we have that if $H^{1}_{p}(\R;X) = W^{1}_{p}(\R;X)$, then $X$ is a UMD space (see \cite{boek}).

For $0 < A < B < \infty$ we define $\Phi_{A,B}(\R^{d})$ as the set of all sequences $\varphi = (\varphi_{n})_{n \in \N} \subset \mathcal{S}(\R^{d};X)$ which can be constructed in the following way: given $\varphi_{0} \in \mathcal{S}(\R^{d})$ with
\[
0 \leq \hat{\varphi} \leq 1, \:\: \hat{\varphi}(\xi) = 1 \:\:\mbox{if $|\xi| \leq A$}, \:\:\hat{\varphi}(\xi) = 0 \:\:\mbox{if $|\xi| \geq B$},
\]
$(\varphi_{n})_{n \geq 1}$ is determined by
\[
\hat{\varphi}_{n} = \hat{\varphi}_{1}(2^{-n+1}\,\cdot\,) = \hat{\varphi}_{0}(2^{-n}\,\cdot\,) - \hat{\varphi}_{0}(2^{-n+1}\,\cdot\,), \quad\quad n \geq 1.
\]
Observe that
\begin{equation}\label{eq:Fourier_support_LP-seq}
\supp \hat{\varphi}_{0} \subset \{ \xi : |\xi| \leq B \} \quad \mbox{and} \quad \supp \hat{\varphi}_{n} \subset \{ \xi : 2^{n-1}A \leq |\xi| \leq 2^{n}B \}, \, n \geq 1.
\end{equation}
We furthermore put $\Phi(\R^{d}) := \bigcup_{0<A<B<\infty}\Phi_{A,B}(\R^{d})$.

Let $\varphi = (\varphi_{n})_{n \in \N} \in \Phi(\R^{d})$. We define the operators $\{S_{n}\}_{n \in \N} \subset  \mathcal{L}(\mathcal{S}'(\R^{d};X),\mathscr{O}_{M}(\R^{d};X))$ by
\[
S_{n}f := \varphi_{n}*f = \mathscr{F}^{-1}[\hat{\varphi}_{n}\hat{f}], \quad\quad f \in \mathcal{S}'(\R^{d};X),
\]
where $\mathscr{O}_{M}(\R^{d};X)$ stands for the space of all $X$-valued slowly increasing smooth functions on $\R^{d}$.
Given $s \in \R$, $p \in [1,\infty)$, $q \in [1,\infty]$ and $w \in A_{\infty}(\R^{d})$, the Triebel-Lizorkin space $F^{s}_{p,q}(\R^{d},w;X)$ is defined as the space of all $f \in \mathcal{S}'(\R^{d};X)$ for which
\[
\norm{f}_{F^{s}_{p,q}(\R^{d},w;X)} := \norm{ (2^{sn}S_{n}f)_{n \in \N} }_{L^{p}(\R^{d},w)[\ell^{q}(\N)](X)} < \infty.
\]
Each choice of $\varphi \in \Phi(\R^{d})$ leads to an equivalent extended norm on $\mathcal{S}'(\R^{d};X)$.

The $H$-spaces are related to the $F$-spaces as follows. In the scalar-valued case $X=\C$, we have
\begin{equation}\label{eq:identity_Bessel-Potential_Triebel-Lizorkin;scalar-valued}
H^{s}_{p}(\R^{d},w) = F^{s}_{p,2}(\R^{d},w), \quad\quad p \in (1,\infty), w \in A_{p}.
\end{equation}
In the unweighted vector-valued case, this identity is valid if and only if $X$ is isomorphic to a Hilbert space.
For general Banach spaces $X$ we still have (see \cite[Proposition~3.12]{Meyries&Veraar_sharp_embd_power_weights})
\begin{equation}\label{eq:relation_Bessel-Potential_Triebel-Lizorkin;Ap}
F^{s}_{p,1}(\R^{d},w;X) \hookrightarrow H^{s}_{p}(\R^{d},w;X) \hookrightarrow F^{s}_{p,\infty}(\R^{d},w;X),
\quad\quad p \in (1,\infty),w \in A_{p}(\R^{d}),
\end{equation}
and
\begin{equation}\label{eq:triebel_into_Lp;A_oneindig}
\left( \mathcal{S}(\R^{d};X),\norm{\,\cdot\,}_{F^{s}_{p,1}(\R^{d},w;X)} \right) \hookrightarrow L^{p}(\R^{d},w;X),
\quad\quad p \in [1,\infty), w \in A_{\infty}.
\end{equation}
For UMD spaces $X$ there is a suitable randomized substitute for \eqref{eq:identity_Bessel-Potential_Triebel-Lizorkin;scalar-valued}:
if $p \in (1,\infty)$ and $w \in A_{p}$, then (see \cite[Proposition~3.2]{Meyries&Veraar_pt-multiplication})
\begin{equation}\label{eq:randomized-LP-decomp}
\norm{f}_{H^{s}_{p}(\R^{d},w;X)} \eqsim \sup_{N \in \N}\normB{\sum_{n=0}^{N}\varepsilon_{n}2^{ns}S_{n}f}_{L^{p}(\Omega;L^{p}(\R^{d},w;X))}, \quad\quad f \in \mathcal{S}'(\R^{d};X).
\end{equation}
Moreover, the implicit constants in \eqref{eq:randomized-LP-decomp} can be taken of the form $C=C_{X,p,d,s}([w]_{A_{p}})$ for some increasing function $C_{X,p,d,s}:[1,\infty) \longra (0,\infty)$ only depending on $X$, $p$, $d$ and $s$.

\subsection{Fourier Multipliers}

Let $X$ be a Banach space. We write $\widehat{L^{1}}(\R^{d};X) := \mathscr{F}^{-1}L^{1}(\R^{d};X) \subset \mathcal{S}'(\R^{d};X)$. For a symbol $m \in L^{\infty}(\R^{d})$ we define the operator $T_{m}$ by
\[
T_{m}: \widehat{L^{1}}(\R^{d};X) \longra \widehat{L^{1}}(\R^{d};X),\, f \mapsto \mathscr{F}^{-1}[m\hat{f}].
\]
Given $p \in [1,\infty)$ and  $w \in A_{\infty}(\R^{d})$, we call $m$ a \emph{Fourier multiplier} on $L^{p}(\R^{d},w;X)$ if $T_{m}$ restricts to an operator on $\widehat{L^{1}}(\R^{d};X) \cap L^{p}(\R^{d},w;X)$ which is bounded with respect to the $L^{p}(\R^{d},w;X)$-norm. In this case $T_{m}$ has a unique extension to a bounded linear operator on $L^{p}(\R^{d},w;X)$ due to the denseness of
$\mathcal{S}(\R^{d};X)$ in $L^{p}(\R^{d},w;X)$, which we still denote by $T_{m}$.
We denote by $\mathcal{M}_{p,w}(X)$ the set of all Fourier multipliers $m \in L^{\infty}(\R^{d})$ on $L^{p}(\R^{d},w;X)$.
Equipped with the norm $\norm{m}_{\mathcal{M}_{p,w}(X)} := \norm{T_{m}}_{\mathcal{B}(L^{p}(\R^{d},w;X))}$, $\mathcal{M}_{p,w}(X)$ becomes a Banach algebra (under the natural pointwise operations) for which the natural inclusion $\mathcal{M}_{p,w}(X) \hookrightarrow \mathcal{B}(L^{p}(\R^{d},w;X)$ is an isometric Banach algebra homomorphism; see \cite{Kunstmann&Weis_lecture_notes} for the unweighted setting.

For each $N \in \N$ we define $\mathscr{M}_{N}(\R^{d})$ as the space of all $m \in C^{N}(\R^{d} \setminus \{0\})$ for which
\[
\norm{m}_{\mathscr{M}_{N}} = \norm{m}_{\mathscr{M}_{N}(\R^{d})}  :=  \sup_{|\alpha| \leq N}\sup_{\xi \neq 0}|\xi|^{|\alpha|}|D^{\alpha}m(\xi)| < \infty.
\]
If $X$ is a UMD Banach space, $p \in (1,\infty)$ and $w \in A_{p}(\R^{d})$, then we have $\mathscr{M}_{d+2}(\R^{d}) \hookrightarrow \mathcal{M}_{p,w}(X)$ with norm $\leq C_{X,p,d}([w]_{A_{p}})$, where $C_{X,p,d}:[1,\infty) \longra (0,\infty)$ is some increasing function only depending on $X$, $d$ and $p$; see \cite[Proposition~3.1]{Meyries&Veraar_pt-multiplication}.

\section{$\mathcal{R}$-Boundedness of Fourier Multipliers}\label{sec:R-bdd_FM}

At several points in the proof of the randomized difference norm characterization from Theorem~\ref{thm:intro_main_result;difference} we need the
$\mathcal{R}$-boundedness of a sequence of Fourier multiplier operators on $L^{p}(\R^{d},w;X)$.
In this section we provide the necessary $\mathcal{R}$-boundedness results.

In many situations, the $\mathcal{R}$-boundedness of a family of operators is proved under the assumption of property $(\alpha)$ (see e.g. \cite{Clement&Pagter&Suckochev&Witvliet_Schauder_decompositions,Girardi&Weis_criteria_R-boundedness,Kunstmann&Weis_lecture_notes,Venni_R-boundedness}).
Concerning operator families on $L^{p}(\T^{d};X)$ or $L^{p}(\R^{d};X)$, the necessity of property~$(\alpha)$ for a number of conclusions of this kind is proved in \cite{Hytonen&Weis_necessity_alpha}.
For example, in the the setting of Fourier multipliers it holds true that every uniform set of Marcinkiewicz multipliers on $\R^{d}$ is $\mathcal{R}$-bounded on $L^{p}(\R^{d};X)$ if and only if $X$ is a UMD space with property~$(\alpha)$. In particular, given a UMD space $X$, in the one-dimensional case $d=1$ one has that $\mathscr{M}_{1}(\R) \hookrightarrow \mathcal{M}_{p,\mathbf{1}_{
\R}}(X)$ maps bounded sets to $\mathcal{R}$-bounded sets if and only if $X$ has property~$(\alpha)$.
Regarding the sufficiency of property $(\alpha)$ for the $\mathcal{R}$-boundedness of Fourier multipliers, in the weighted setting we have:

\begin{prop}\label{prop:R-bdd_FM_alpha}
Let $X$ be a UMD space with property $(\alpha)$ and $p \in (1,\infty)$.
\begin{itemize}
\item[(i)] For all weights $w \in A_{p}(\R^{d})$, $\mathscr{M}_{d+2}(\R^{d}) \hookrightarrow \mathcal{M}_{p,w}(X)$ maps bounded sets to $\mathcal{R}$-bounded sets.
\item[(ii)] Let $w \in A_{p}^{rec}(\R^{d})$, i.e. $w$ is a locally integrable weight on $\R^{d}$ which is uniformly $A_{p}$ in each of the coordinates separately; see \cite{Kurtz_weighted_fm}. Write $\R^{d}_{*} = [\R \setminus \{0\}]^{d}$.
    If $\mathscr{M} \subset L^{\infty}(\R^{d}) \cap C^{d}(\R^{d}_{*})$ satisfies
\[
C_{\mathscr{M}} := \sup_{M \in \mathscr{M}}\sup_{\alpha \leq \mathbf{1}}\sup_{ \xi \in \R^{d}_{*}} |\xi^{\alpha}|\,|D^{\alpha}m(\xi)| < \infty,
\]
then $\mathscr{M}$ defines an $\mathcal{R}$-bounded collection of Fourier multiplier operators $\mathcal{T}_{\mathscr{M}} = \{ T_{M} : M \in \mathscr{M} \}$ in $\mathcal{B}(L^{p}(\R^{d},w;X))$ with $\mathcal{R}(\mathcal{T}_{\mathscr{M}}) \lesssim_{X,p,d,w} C_{\mathscr{M}}$.
\end{itemize}
\end{prop}
\begin{proof}
(i) Let $w \in A_{p}$. For each $N \in \N$ we define $\mathscr{RM}_{N}(\R^{d};\mathcal{B}(X))$ as the space of all operator-valued symbols $m \in C^{N}(\R^{d} \setminus \{0\};\mathcal{B}(X))$ for which
\[
\norm{m}_{\mathscr{RM}_{N}} = \norm{m}_{\mathscr{RM}_{N}(\R^{d};\mathcal{B}(X))}
:=  \mathcal{R}\big\{\,|\xi|^{|\alpha|}D^{\alpha}m(\xi) : \xi \neq 0, |\alpha| \leq N \,\big\} < \infty.
\]
If $Y$ is a UMD space, then $\mathscr{RM}_{d+2}(\R^{d};\mathcal{B}(Y)) \hookrightarrow \mathcal{M}_{p,w}(Y)$ (as remarked before \cite[Proposition~3.1]{Meyries&Veraar_pt-multiplication}). Using this for $Y = \mathrm{Rad}(X)$, the desired result follows  in the same spirit as in
\cite[Section~3]{Girardi&Weis_criteria_R-boundedness} (also see \cite{Hytonen_survey-paper_vector-valued,Kunstmann&Weis_lecture_notes}).

(ii) Put $I_{j}:= [-2^{j},-2^{j-1}) \cup (2^{j-1},2^{j}]$ for each $j \in \Z$. For each $k \in \{1,\ldots,d\}$ it can be shown that $\{1_{\R^{k} \times I_{j} \times \R^{d-k}}\}_{j \in \Z} \subset \mathcal{M}_{p,w}(X)$ and that the associated sequence of Fourier multiplier operators $\{\Delta_{k}[I_{j}]\}_{j \in \Z}$ defines an unconditional Schauder decomposition of $L^{p}(\R^{d},w;X)$; see e.g. \cite[Chapter~4]{Lindemulder_master-thesis}.
Since $\{\Delta_{k}[I_{j}]\}_{j \in \Z}$ and $\{\Delta_{l}[I_{j}]\}_{j \in \Z}$ commute for $k,l \in \{1,\ldots,d\}$ and since $X$ is assumed to have property $(\alpha)$, it follows (see \cite[Remark~2.5.2]{Witvliet_PHD-thesis}) that the product decomposition $\{\prod_{i=1}^{d}\Delta_{k}[I_{j}] \}$ is an unconditional Schauder decomposition of $L^{p}(\R^{d},w;X)$. One can now proceed as in the unweighted case; see e.g. \cite[Theorem~4.13$\&$Example~5.2]{Kunstmann&Weis_lecture_notes}.
\end{proof}

As we will see below, for general UMD spaces it is still possible to give criteria for the $\mathcal{R}$-boundedness of a sequence of Fourier multipliers. Before we go to the Fourier analytic setting, we start with a general proposition which serves as the main tool for the $\mathcal{R}$-boundedness of Fourier multipliers below.
In order to state the proposition, we first need to introduce some notation.

Let $Y$ be a Banach space. For a sequence $\{T_{j}\}_{j \in \N} \subset \mathcal{B}(Y)$ we write
\[
\norm{\{T_{j}\}_{j \in \N}}_{Y \to \mathrm{Rad}(Y)} := \inf\left\{ C : \normB{\sum_{j=0}^{n}\varepsilon_{j}T_{j}y}_{L^{2}(\Omega;Y)} \leq C\norm{y}_{Y}, y \in Y \right\}
\]
and
\[
\norm{\{T_{j}\}_{j \in \N}}_{\mathrm{Rad}(Y) \to Y} :=
\inf\left\{ C : \normB{\sum_{j=0}^{n}T_{j}y_{j}}_{Y} \leq C\normB{\sum_{j=0}^{n}\varepsilon_{j}y_{j}}_{L^{2}(\Omega;Y)}, n \in \N, y_{0},\ldots,y_{n} \in Y \right\}.
\]
In the following remark we provide an interpretation of these quantities in terms of the space $\mathrm{Rad}(Y)$, which gives a motivation for the chosen notation.

\begin{remark}
Identifying $\{T_{j}\}_{j \in \N}$ with the linear operator $\mathbf{T}:Y \longra \ell^{0}(\N;X), y \mapsto (T_{j}y)_{j \in \N}$, we have
\[
\norm{\{T_{j}\}}_{Y \to \mathrm{Rad}(Y)} = \norm{\{T_{j}\}}_{\mathcal{B}(Y,\mathrm{Rad}(Y))} = \norm{\mathbf{T}}_{\mathcal{B}(Y,\mathrm{Rad}(Y))},
\]
where $\norm{\,\cdot\,}_{\mathcal{B}(Y,\mathrm{Rad}(Y))}$ is, in the natural way, viewed as an extended norm on $L(Y,\ell^{0}(\N;X))$, the space of linear operators from $Y$ to $\ell^{0}(\N;X)$.
Similarly, identifying $\{T_{j}\}_{j \in \N}$ with the linear operator $\mathbf{T}^{\mathrm{t}}:c_{00}(\N;X) \longra Y, (y_{j})_{j \in \N} \mapsto \sum_{j \in \N}T_{j}y_{j}$, we have
\[
\norm{\{T_{j}\}}_{\mathrm{Rad}(Y) \to Y} = \norm{\{T_{j}\}}_{\mathcal{B}(\mathrm{Rad}(Y),Y)} = \norm{\mathbf{T}^{\mathrm{t}}}_{\mathcal{B}(\mathrm{Rad}(Y),Y)},
\]
where $\norm{\,\cdot\,}_{\mathcal{B}(\mathrm{Rad}(Y),Y)}$ is viewed, in the natural way, as an extended norm on $L(c_{00}(Y),Y)$.

Using that the natural map $i:\mathrm{Rad}(Y^{*}) \longra \mathrm{Rad}(Y)^{*}$ is a contraction (see \cite{boek2}), we find that
\begin{eqnarray*}
\norm{\{T_{j}\}}_{\mathrm{Rad}(Y) \to Y}
&=& \norm{\mathbf{T}^{\mathrm{t}}}_{\mathcal{B}(\mathrm{Rad}(Y),Y)} = \norm{(\mathbf{T}^{t})^{*}}_{\mathcal{B}(Y^{*},\mathrm{Rad}(Y)^{*})}
 = \norm{i \circ (\{T_{j}^{*}\})^{t}}_{\mathcal{B}(Y^{*},\mathrm{Rad}(Y)^{*})} \\
&\leq& \norm{(\{T_{j}^{*}\})^{t}}_{\mathcal{B}(Y^{*},\mathrm{Rad}(Y^{*}))} = \norm{\{T_{j}^{*}\}}_{Y^{*} \to \mathrm{Rad}(Y^{*})}.
\end{eqnarray*}
If $X$ is $K$-convex with $K$-convexity constant $K_{X}$,\footnote{For the definition of $K$-convexity we refer to \cite{boek,Maurey_Type_cotype_K-convexity}. All UMD spaces are $K$-convex.} then $i$ is an isomorphism of Banach spaces with $\norm{i^{-1}} \leq K_{X}$ (see \cite{boek2}), so that
\begin{eqnarray*}
\norm{\{T_{j}\}}_{Y \to \mathrm{Rad}(Y)}
&=& \norm{\mathbf{T}}_{\mathcal{B}(Y,\mathrm{Rad}(Y))} = \norm{\mathbf{T}^{*}}_{\mathcal{B}(\mathrm{Rad}(Y)^{*},Y^{*})}
 = \norm{\{T_{j}^{*}\} \circ i^{-1}}_{\mathcal{B}(\mathrm{Rad}(Y)^{*},Y^{*})} \\
&\leq& K_{X}\norm{\{T_{j}^{*}\} }_{\mathcal{B}(\mathrm{Rad}(Y^{*}),Y^{*})} = K_{X}\norm{\{T_{j}^{*}\}}_{\mathrm{Rad}(Y^{*}) \to Y^{*}}.
\end{eqnarray*}
\end{remark}

\begin{prop}\label{prop:KW_thm3.3}
Let $Y$ be a Banach space and let $\{U_{j}\}_{j \in \N}$ and $\{V_{j}\}_{j \in \N}$ be two sequences of operators in $\mathcal{B}(Y)$.
\begin{itemize}
\item[(i)] The following inequalities hold true:
\begin{equation}\label{eq:prop:KW_thm3.3;(i);1e_verg}
\mathcal{R}(\{U_{j}\}) \leq \norm{\{U_{j}\}}_{\mathrm{Rad}(Y) \to Y} \leq \norm{\{U_{j}\}}_{\mathrm{Rad}(\mathcal{B}(Y))} \leq \sup_{n}\sup_{\epsilon_{j}=\pm1}\normB{\sum_{j=0}^{n}\epsilon_{j}U_{j}},
\end{equation}
\begin{equation}\label{eq:prop:KW_thm3.3;(i);3e_verg}
\mathcal{R}(\{U_{j}\}) \leq \norm{\{U_{j}\}}_{Y \to \mathrm{Rad}(Y)} \leq \norm{\{U_{j}\}}_{\mathrm{Rad}(\mathcal{B}(Y))} \leq \sup_{n}\sup_{\epsilon_{j}=\pm1}\normB{\sum_{j=0}^{n}\epsilon_{j}U_{j}}
\end{equation}
and
\begin{equation}\label{eq:prop:KW_thm3.3;(i);2e_verg}
\norm{\{U_{j}V_{j}\}_{j \in \N}}_{\mathrm{Rad}(Y) \to Y} \leq
\norm{\{U_{j}\}_{j \in \N}}_{\mathrm{Rad}(\mathcal{B}(Y))}\mathcal{R}(\{V_{j}\}_{j \in \N}).
\end{equation}
\item[(ii)] Suppose that $E$ has property $(\Delta)$. If
\[
C_{1} := \norm{\{U_{j}\}_{j \in \N}}_{\mathrm{Rad}(Y) \to Y} < \infty \quad \mbox{and} \quad C_{2}:= \norm{\{V_{j}\}_{j \in \N}}_{Y \to \mathrm{Rad}(Y)} < \infty,
\]
then $\{\sum_{j=0}^{n}U_{k}V_{k}\}$ is $\mathcal{R}$-bounded with $\mathcal{R}$-bound $\leq \Delta_{E}C_{1}C_{2}$.
\end{itemize}
\end{prop}
\begin{proof}
Except for \eqref{eq:prop:KW_thm3.3;(i);3e_verg}, where we follow the estimates from the proof of \cite[Lemma~4.1]{Meyries&Veraar_pt-multiplication},
the proposition follows easily by inspection of the proof of \cite[Theorem~3.3]{Kalton&Weis_sums_of_closed_operators}.
Let us provide the details for the convenience of the reader.

(i) The third inequality in \eqref{eq:prop:KW_thm3.3;(i);1e_verg} is trivial and the second inequality in \eqref{eq:prop:KW_thm3.3;(i);1e_verg} is just the inequality \eqref{eq:prop:KW_thm3.3;(i);2e_verg} with $V_{j}=I$ for all $j$. For the first inequality in \eqref{eq:prop:KW_thm3.3;(i);1e_verg}, let $y_{0},\ldots,y_{n} \in Y$.
For every $\{\epsilon_{j}\}_{j \in \N} \in \{-1,1\}^{n+1}$ we have
\[
\normB{\sum_{j=0}^{n}\epsilon_{j}U_{j}y_{j}}_{Y} \leq \norm{\{U_{j}\}_{j \in \N}}_{\mathrm{Rad}(Y) \to Y}\normB{\sum_{j=0}^{n}\varepsilon_{j}y_{j}}_{L^{2}(\Omega;Y)}
\]
because $\{\varepsilon_{j}\}_{j=0}^{n}$ and $\{\epsilon_{j}\varepsilon_{j}\}_{j=0}^{n}$ are identically distributed.
Plugging in $\epsilon_{j} = \varepsilon_{j}(\omega)$ and taking $L^{2}$-norms with respect to $\omega \in \Omega$, the desired inequality follows.

In \eqref{eq:prop:KW_thm3.3;(i);3e_verg} we only need to prove the first inequality; the other two inequalities are trivial.
For this we use the fact \cite[Lemma~3.12]{Girardi&Weis_FM_geometry_BS} that for any $\{y_{j,k}\}_{j,k=0}^{n} \subset Y$ one has the inequality
\begin{equation}\label{eq:prop:KW_thm3.3;(i);3e_verg;proof;ineq_KW_uit_MV}
\normB{\sum_{j=0}^{n}\varepsilon_{j}y_{j,j}}_{L^{2}(\Omega;Y)} \leq
\normB{\sum_{j,k=0}^{n}\varepsilon_{j}\varepsilon_{k}'y_{j,k}}_{L^{2}(\Omega \times \Omega';Y)}.
\end{equation}
Now let $y_{0},\ldots,y_{n} \in Y$.
Denote by $\{\tilde{U}_{j}\} \subset \mathcal{B}(L^{2}(\Omega;Y))$ the sequence of operators pointwise induced by $\{U_{j}\}$. Using Fubini one easily sees that $\norm{\{\tilde{U}_{j}\} }_{L^{2}(\Omega;Y) \to \mathrm{Rad}(L^{2}(\Omega;Y))} \leq \norm{\{U_{j}\}}_{Y \to \mathrm{Rad}(Y)}$.
Invoking \eqref{eq:prop:KW_thm3.3;(i);3e_verg;proof;ineq_KW_uit_MV} with $y_{j,k} = U_{k}y_{j}$, we thus find
\begin{eqnarray*}
\normB{\sum_{j=0}^{n}\varepsilon_{j}U_{j}y_{j}}_{L^{2}(\Omega;Y)}
&\leq& \normB{\sum_{j,k=0}^{n}\varepsilon_{j}\varepsilon_{k}'U_{k}y_{j}}_{L^{2}(\Omega \times \Omega';Y)} \\
&=& \normB{\sum_{k=0}^{n}\varepsilon_{k}'\tilde{U}_{k}\Big(\sum_{j=0}^{n} \varepsilon_{j}y_{j}\Big)}_{L^{2}(\Omega';L^{2}(\Omega;Y))} \\
&\leq& \norm{\{U_{j}\}}_{Y \to \mathrm{Rad}(Y)}\normB{\sum_{j=0}^{n} \varepsilon_{j}y_{j}}_{L^{2}(\Omega;Y))}.
\end{eqnarray*}

For \eqref{eq:prop:KW_thm3.3;(i);2e_verg} note that if $y_{0},\ldots,y_{n} \in Y$, then
\begin{eqnarray*}
\normB{\sum_{j=0}^{n}U_{j}V_{j}y_{j}}_{Y}
&=& \norml{\, \E\left[ \left( \sum_{j=0}^{n}\varepsilon_{j}U_{j} \right)\,\left( \sum_{j=0}^{n}\varepsilon_{j}V_{j}y_{j} \right) \right] \,}_{Y} \\
&\leq& \normB{\sum_{j=0}^{n}\varepsilon_{j}U_{j}}_{L^{2}(\Omega;\mathcal{B}(Y))} \normB{\sum_{j=0}^{n}\varepsilon_{j}V_{j}y_{j}}_{L^{2}(\Omega;Y)} \\
&\leq& \norm{\{U_{j}\}}_{\mathrm{Rad}(\mathcal{B}(Y))}\mathcal{R}(\{V_{j}\})\normB{\sum_{j=0}^{n}\varepsilon_{j}y_{j}}_{L^{2}(\Omega;Y)}.
\end{eqnarray*}

(ii) Write $S_{k}:=\sum_{j=0}^{k}U_{j}V_{j}$ for each $k \in \N$. For all $y_{0},\ldots,y_{n} \in Y$ we have
\begin{eqnarray*}
\normB{\sum_{k=0}^{n}\varepsilon_{k}S_{k}y_{k}}_{L^{2}(\Omega;Y)}
&=& \normB{\sum_{j=0}^{n}U_{j}\sum_{k=j}^{n}\varepsilon_{k}V_{j}y_{k}}_{L^{2}(\Omega;Y)} \\
&\leq& C_{1}\normB{\sum_{j=0}^{n}\varepsilon'_{j}\sum_{k=j}^{n}\varepsilon_{k}V_{j}y_{k}}_{L^{2}(\Omega;L^{2}(\Omega';Y))} \\
&\leq& \Delta_{Y}C_{1}\normB{\sum_{j=0}^{n}\varepsilon'_{j}V_{j}\sum_{k=0}^{n}\varepsilon_{k}y_{k}}_{L^{2}(\Omega;L^{2}(\Omega';Y))} \\
&\leq& \Delta_{Y}C_{1}C_{2}\normB{\sum_{k=0}^{n}\varepsilon_{k}y_{k}}_{L^{2}(\Omega;;Y)},
\end{eqnarray*}
which proves the required $\mathcal{R}$-bound.
\end{proof}

For later reference it will be convenient to record the following immediate corollary to the estimates \eqref{eq:prop:KW_thm3.3;(i);1e_verg} and \eqref{eq:prop:KW_thm3.3;(i);3e_verg}  in (i) of the above proposition:
\begin{cor}\label{prop:R-bdd_via_Mikhlin}
Let $X$ be a Banach space, $p \in (1,\infty)$ and $w \in A_{p}(\R^{d})$.
Let $\{m_{j}\}_{j \in \N} \subset \mathcal{M}_{p,w}(X)$ be a sequence of symbols such that
\begin{equation}\label{eq:prop:R-bdd_via_Mikhlin;cond_pm1}
K := \sup_{n}\sup_{\epsilon_{j}=\pm1}\normB{\sum_{j=0}^{n}\epsilon_{j}m_{j}}_{\mathcal{M}_{p,w}(X)} < \infty.
\end{equation}
Then $\{m_{j}\}_{j \in \N}$ defines an $\mathcal{R}$-bounded sequence of Fourier multiplier operators $\{T_{m_{j}}\}_{j \in \N}$ on $Y=L^{p}(\R^{d},w;X)$ with $\mathcal{R}$-bound
\[
\mathcal{R}(\{T_{m_{j}}\}) \leq \norm{\{T_{j}\}}_{\mathrm{Rad}(Y) \to Y} \vee  \norm{\{T_{j}\}}_{Y \to \mathrm{Rad}(Y)} \leq \norm{\{T_{j}\}}_{\mathrm{Rad}(\mathcal{B}(Y))}  \leq K.
\]
\end{cor}

If $X$ is a UMD space, $p \in (1,\infty)$ and $w \in A_{p}(\R^{d})$,
then we have $\mathscr{M}_{d+2}(\R^{d}) \hookrightarrow \mathcal{M}_{p,w}(X)$. So the number $K$ from \eqref{eq:prop:R-bdd_via_Mikhlin;cond_pm1} can be explicitly bounded via the Mihlin condition defining $\mathscr{M}_{d+2}(\R^{d})$. In particular, for a bounded sequence in $\mathscr{M}_{d+2}(\R^{d})$ which is locally finite in a uniform way we find:

\begin{cor}\label{prop;R-bddness_fm;nbhd_cond_non-zero_restr}
Let $X$ be a UMD space, $p \in (1,\infty)$ and $w \in A_{p}(\R^{d})$. Let $\{m_{j}\}_{j \in \N} \subset L^{\infty}(\R^{d})$ be a sequence of symbols such that:
\begin{itemize}
\item[(a)] There exists $N \in \N$ such that every $\xi \in \R^{d} \setminus \{0\}$ possesses an open neighborhood $U \subset \R^{d} \setminus \{0\}$ with the property that $\#\{ j : m_{j}|_{U} \neq 0 \} \leq N$.
\item[(b)] $\{m_{j}\}_{j \in \N}$ is a bounded sequence in $\mathscr{M}_{d+2}(\R^{d})$.
\end{itemize}
Then $\{m_{j}\}_{j \in \N}$ defines an $\mathcal{R}$-bounded sequence of Fourier multiplier operators $\{T_{m_{j}}\}_{j \in \N}$
 on $L^{p}(\R^{d},w;X)$ with $\mathcal{R}$-bound
\[
\mathcal{R}(\{T_{m_{j}}\}) \leq \sup_{n}\sup_{\epsilon_{j}=\pm1}\normB{\sum_{j=0}^{n}\epsilon_{j}m_{j}}_{\mathcal{M}_{p,w}(X)}
\lesssim C_{X,p,d}([w]_{A_{p}})N\sup_{j \in \N}\norm{m_{j}}_{\mathscr{M}_{d+2}},
\]
where $C_{X,p,d}:[1,\infty) \longra (0,\infty)$ is some increasing function only depending on $X$, $p$ and $d$.
\end{cor}

An example for the 'uniform locally finiteness condition' (a) from the above corollary is a kind of dyadic corona condition on the supports of the symbols:
\begin{ex}\label{ex:prop;R-bddness_fm;nbhd_cond_non-zero_restr;ex_cond(a);dyadic_support}
Suppose that $\{m_{j}\}_{j \in \N} \subset L^{\infty}(\R^{d})$ satisfies the support condition
\begin{equation}\label{eq:ex:prop;R-bddness_fm;nbhd_cond_non-zero_restr;ex_cond(a);dyadic_support;condition}
\supp m_{0} \subset \{ \xi : |\xi| \leq c \} \quad \mbox{and} \quad \supp m_{j} \subset \{ \xi :  c3^{-1}2^{j-J+1} \leq |\xi| \leq c2^{j} \},\, j \geq 1,
\end{equation}
for some $c>0$ and $J \in \Z_{>0}$. Then $\supp m_{j} \cap \supp m_{k} = \emptyset$ for all $j,k \in \N$ with $|j-k| \geq J+1$.
In particular, condition (a) of Corollary~\ref{prop;R-bddness_fm;nbhd_cond_non-zero_restr} is satisfied with $N=J$.
\end{ex}

\begin{ex}\label{ex:prop;R-bddness_fm;nbhd_cond_non-zero_restr;ex_cond_(a)en(b)}
Suppose that $m_{0} \in C^{d+2}_{c}(\R^{d})$ and $m_{1} \in C^{d+2}_{c}(\R^{d} \setminus \{0\})$. Set $m_{j} := m(2^{-j}\,\cdot\,)$ for each $j \geq 2$. Then $\{m_{j}\}_{j \in \N}$ fulfills the conditions (a) and (b) of Corollary~\ref{prop;R-bddness_fm;nbhd_cond_non-zero_restr}, where (a) follows from Example~\ref{ex:prop;R-bddness_fm;nbhd_cond_non-zero_restr;ex_cond(a);dyadic_support} and (b) from the dilation invariance of the Mihlin condition defining $\mathscr{M}_{d+2}(\R^{d})$. In particular, given $\varphi = \{\varphi_{j}\}_{j \in \N} \in \Phi(\R^{d})$, Corollary~\ref{prop;R-bddness_fm;nbhd_cond_non-zero_restr} can be applied to the sequence of symbols $\{m_{j}\}_{j \in \N} = \{\hat{\varphi}_{j}\}_{j \in \N}$, whose associated sequence of Fourier multiplier operators is $\{S_{j}\}_{j \in \N}$.
\end{ex}

Up to now we have only exploited Proposition~\ref{prop:KW_thm3.3}(i) in order to get $\mathcal{R}$-boundedness of a sequence of Fourier multipliers.
However, in many situations the condition \eqref{eq:prop:R-bdd_via_Mikhlin;cond_pm1} is too strong. It is for example not fulfilled by the sequence $\{m_{j} = m(2^{-j}\,\cdot\,) \}_{j \in \N}$, where $m \in C^{\infty}_{c}(\R^{d})$ is a given symbol which is non-zero in the origin; this follows from the fact that $\mathcal{M}_{p,w}(X) \hookrightarrow L^{\infty}(\R^{d})$.
The case that $m$ is constant on a neighborhood of the origin can be handled by the following proposition (see Corollary~\ref{cor:prop:R-bdd_partial_sums;dilations_constant_on_nbhd_origin}), of which the main ingredient is Proposition~\ref{prop:KW_thm3.3}(ii):

\begin{prop}\label{prop:R-bdd_partial_sums}
Let $X$ be a UMD space, $p \in (1,\infty)$ and $w \in A_{p}(\R^{d})$. Let $\{m_{j}\}_{j \in \N} \subset \mathcal{M}_{p,w}(X)$ be a sequence of Fourier multiplier symbols which satisfies the support condition \eqref{eq:ex:prop;R-bddness_fm;nbhd_cond_non-zero_restr;ex_cond(a);dyadic_support;condition} for some $c>0$ and $J \in \N$. Write $T_{j}=T_{m_{j}}$ for the Fourier multiplier operator on $Y=L^{p}(\R^{d},w;X)$ associated with $m_{j}$ for each $j \in \N$.
If
\begin{equation}\label{eq:prop:R-bdd_partial_sums;condition}
K:= \norm{ \{T_{j}\} }_{\mathrm{Rad}(Y)\to Y} \wedge \norm{ \{T_{j}\} }_{Y \to \mathrm{Rad}(Y)}  < \infty,
\end{equation}
then the collection of partial sums $\{ \sum_{j=0}^{n}T_{j} : n \in \N \}$ is $\mathcal{R}$-bounded with $\mathcal{R}$-bound $\leq (2J+1)C_{X,p,d}([w]_{A_{p}})\,K$ for some increasing function $C_{X,p,d}:[1,\infty) \longra (0,\infty)$ only depending on $X$, $p$ and $d$.
\end{prop}

\begin{proof}
Due to scaling invariance of the $A_{p}$-characteristic, we may without loss of generality assume that $c=\frac{3}{2}$.
Fix $\varphi = (\varphi_{j})_{j \in \N} \in \Phi_{1,\frac{3}{2}}(\R^{d})$ and denote by $\{S_{j}\}_{j \in \N}$ the corresponding convolution operators. For convenience of notation we put $\varphi_{j}:= 0$ and  $S_{j}:=0$ for every $j \in \Z_{<0}$. For each $j \in \N$ we define $R_{j} := \sum_{\ell=-J}^{J}S_{j+\ell}$. By Example~\ref{ex:prop;R-bddness_fm;nbhd_cond_non-zero_restr;ex_cond_(a)en(b)} (and Corollary~\ref{prop;R-bddness_fm;nbhd_cond_non-zero_restr}), there exists an increasing function $\tilde{C}_{X,p,d}:[1,\infty) \longra (0,\infty)$, only depending on $X$, $p$ and $d$, such that
\[
\norm{ \{S_{j}\} }_{\mathrm{Rad}(Y)\to Y} \vee \norm{ \{S_{j}\} }_{Y \to \mathrm{Rad}(Y)} \stackrel{\eqref{eq:prop:KW_thm3.3;(i);1e_verg},\eqref{eq:prop:KW_thm3.3;(i);3e_verg}}{\leq} \tilde{C}_{X,p,d}([w]_{A_{p}}),
\]
and thus
\begin{equation}\label{eq:prop:R-bdd_partial_sums;bound_tilde_Sj}
\norm{ \{R_{j}\} }_{\mathrm{Rad}(Y)\to Y} \vee \norm{ \{R_{j}\} }_{Y \to \mathrm{Rad}(Y)} \leq (2J+1)\tilde{C}_{X,p,d}([w]_{A_{p}}).
\end{equation}
As a consequence of the support condition \eqref{eq:ex:prop;R-bddness_fm;nbhd_cond_non-zero_restr;ex_cond(a);dyadic_support;condition} and the fact that
\[
\sum_{\ell=-J}^{J}\hat{\varphi}_{\ell}(\xi) = 1 \:\: \mbox{for} \:\: |\xi| \leq \frac{3}{2} \quad \mbox{and} \quad
\sum_{\ell=-J}^{J}\hat{\varphi}_{j+\ell}(\xi) = 1 \:\: \mbox{for} \:\: 2^{j-J} \leq |\xi| \leq \frac{3}{2}2^{j},\, j \geq 1,
\]
we have $T_{j}R_{j} = R_{j}T_{j}=T_{j}$ for every $j \in \N$.
Since $\{T_{j}\}$ and $\{R_{j}\}$ are commuting and since $\Delta_{Y} \eqsim_{p} \Delta_{Y,p}=\Delta_{X,p} < \infty$ ($X$ being a UMD space), the required $\mathcal{R}$-bound follows from an application of Proposition~\ref{prop:KW_thm3.3}(ii) with either $U_{j} = T_{j}$ and $V_{j}=R_{j}$ or $U_{j}=R_{j}$ and $V_{j}=T_{j}$.
\end{proof}

\begin{remark}
The condition \eqref{eq:prop:R-bdd_partial_sums;condition} in Proposition~\ref{prop:R-bdd_partial_sums} may be replaced by the condition that $\{T_{j}\}$ is $\mathcal{R}$-bounded with $\mathcal{R}$-bound $K$: under this modification, it can be shown that the collection of partial sums is $\mathcal{R}$-bounded with $\mathcal{R}$-bound $\leq (2J+1)^{2}C_{X,p,d}([w]_{A_{p}})\,K$ for some increasing function $C_{X,p,d}:[1,\infty) \longra (0,\infty)$ only depending on $X$, $p$ and $d$. Indeed, in the notation of the proof above, we have
\begin{eqnarray*}
\norm{ \{T_{j}\} }_{\mathrm{Rad}(Y)\to Y}
&=& \norm{ \{R_{j}T_{j}\} }_{\mathrm{Rad}(Y)\to Y} \stackrel{\eqref{eq:prop:KW_thm3.3;(i);2e_verg}}{\leq} \norm{\{R_{j}\}}_{\mathrm{Rad}(\mathcal{B}(Y))} \mathcal{R}(\{T_{j}\}) \\
&\stackrel{\eqref{eq:prop:R-bdd_partial_sums;bound_tilde_Sj}}{\leq}&  (2J+1)\tilde{C}_{X,p,d}([w]_{A_{p}})\mathcal{R}(\{T_{j}\}).
\end{eqnarray*}

An alternative approach for the $\mathcal{R}$-boundedness condition would be to modify the proof of \cite[Theorem~3.9]{Clement&Pagter&Suckochev&Witvliet_Schauder_decompositions} (or \cite[Theorem~2.4.3]{Witvliet_PHD-thesis}), which is a generalization of the vector-valued Stein inequality to the setting of unconditional Schauder decompositions.
Via this approach one would get linear dependence on $J$ instead of quadratic.
\end{remark}

\begin{cor}\label{cor:prop:R-bdd_partial_sums;dilations_constant_on_nbhd_origin}
Let $X$ be a UMD space, $p \in (1,\infty)$ and $w \in A_{p}(\R^{d})$. Suppose that $M \in C^{d+2}_{c}(\R^{d})$ is constant on a neighborhood of $0$ and put $M_{j}:=M(2^{-j}\,\cdot\,)$ for each $j \in \Z$. Then $\{M_{j}\}_{j \in \Z}$ defines an $\mathcal{R}$-bounded sequence of Fourier multiplier operators $\{T_{M_{j}}\}_{j \in \Z}$ in $\mathcal{B}(L^{p}(\R^{d},w;X))$ with $\mathcal{R}$-bound $\lesssim_{M}C_{X,p,d}([w]_{A_{p}})$, where $C_{X,p,d}$ is the function from Proposition~\ref{prop:R-bdd_partial_sums}.
\end{cor}
\begin{proof}
By the scaling invariance of the $A_{p}$-characteristic, it suffices to prove the $\mathcal{R}$-boundedness statement for $\{M_{j}\}_{j \in \N}$ instead of $\{M_{j}\}_{j \in \Z}$. Indeed, for each $K \in \Z_{<0}$ we then in particular have that $\{M_{j}\}_{j \in \N}$ defines an $\mathcal{R}$-bounded sequence of Fourier multiplier operators $\{T_{M_{j}}\}_{j \in \N}$ in $\mathcal{B}(L^{p}(\R^{d},w(2^{-K}\,\cdot\,);X))$ with $\mathcal{R}$-bound $\lesssim_{M}C_{X,p,d}([w]_{A_{p}})$, or equivalently, that $\{M_{j}\}_{j \geq K}$ defines an $\mathcal{R}$-bounded sequence of Fourier multiplier operators $\{T_{M_{j}}\}_{j \geq K}$ in $\mathcal{B}(L^{p}(\R^{d},w(2^{K}\,\cdot\,);X))$ with $\mathcal{R}$-bound $\lesssim_{M}C_{X,p,d}([w]_{A_{p}})$.

Define the sequence of symbols $\{m_{j}\}_{j \in \N}$ by $m_{0}:=M$, $m_{1}:= m_{0}(2^{-1}\,\cdot\,) - m_{0}$, and $m_{j}:= m_{1}(2^{-j+1}\,\cdot\,)$ for $j \geq 2$. Then $\{m_{j}\}_{j \in \N}$ is a bounded sequence in $\mathscr{M}_{d+2}$ which satisfies the support condition \eqref{eq:ex:prop;R-bddness_fm;nbhd_cond_non-zero_restr;ex_cond(a);dyadic_support;condition}. By a combination of Corollary~\ref{prop;R-bddness_fm;nbhd_cond_non-zero_restr}, Example~\ref{ex:prop;R-bddness_fm;nbhd_cond_non-zero_restr;ex_cond(a);dyadic_support} and Proposition~\ref{prop:R-bdd_partial_sums}, the collection of partial sums $\{ T_{M_{i}} : i \in \N \} = \{\sum_{k=0}^{i}T_{m_{k}} : i \in \N \}$ is $\mathcal{R}$-bounded in $\mathcal{B}(L^{p}(\R^{d},w;X))$ (with the required dependence of the $\mathcal{R}$-bound).
\end{proof}

With the following theorem we can in particular treat dilations of symbols $M$ belonging to the Schwartz class $\mathcal{S}(\R^{d})$ without any further restrictions. Note that this would be immediate from Proposition~\ref{prop:R-bdd_FM_alpha}(i) in case of property $(\alpha)$.

\begin{thm}\label{thm:R-bdd_dilations_FM;general_UMD}
Let $X$ be a UMD space, $p \in (1,\infty)$ and $w \in A_{p}(\R^{d})$.
Let $M \in C(\R^{d}) \cap C^{d+2}(\R^{d} \setminus \{0\})$ and set $M_{j}:=M(2^{-j}\,\cdot\,)$ for each $j \in \Z$.
Suppose that there exist $\delta_{0},\delta_{\infty}>0$ such that
\begin{equation}\label{thm:R-bdd_dilations_FM;general_UMD;cond0}
C_{0}  := \sup_{0<|\xi|\leq 1}|\xi|^{-\delta_{0}}|M(\xi)-M(0)| \:\vee\: \sup_{1 \leq |\alpha| \leq d+2}\sup_{0<|\xi| \leq 1}|\xi|^{|\alpha|-\delta_{0}}\big|D^{\alpha}M(\xi)\big| < \infty
\end{equation}
and
\begin{equation}\label{thm:R-bdd_dilations_FM;general_UMD;cond_infty}
C_{\infty} := \sup_{|\alpha| \leq d+2}\sup_{|\xi| \geq 1}|\xi|^{|\alpha|+\delta_{\infty}}|D^{\alpha}M(\xi)| < \infty.
\end{equation}
Then $\{M_{j}\}_{j \in \Z}$ defines an $\mathcal{R}$-bounded sequence of Fourier multiplier operators $\{T_{M_{j}}\}_{j \in \Z}$ in $\mathcal{B}(L^{p}(\R^{d},w;X))$ with $\mathcal{R}$-bound $\leq C_{X,d,p,\delta_{0},\delta_{\infty}}([w]_{A_{p}})[\norm{M}_{\infty} \vee C_{0} \vee C_{\infty}]$, where $C_{X,d,p,\delta_{0},\delta_{\infty}}:[1,\infty) \longra (0,\infty)$ is some increasing function only depending on $X$, $p$, $d$, $\delta_{0}$ and $\delta_{\infty}$.
\end{thm}

\begin{remark}\label{rmk:thm:R-bdd_dilations_FM;general_UMD}
In the proof of Theorem~\ref{thm:R-bdd_dilations_FM;general_UMD} we use the Mihlin multiplier theorem $\mathscr{M}_{d+2} \hookrightarrow \mathcal{M}_{p,w}(X)$. The availability of better multiplier theorems would lead to weaker conditions on $M$.
For example, using the classical Mihlin multiplier condition $|D^{\alpha}m| \lesssim |\xi|^{|\alpha|}$, $\alpha \in \{0,1\}^{d}$, we could treat symbols $M \in C(\R^{d}) \cap C^{d}(\R^{d} \setminus \{0\})$ satisfying \eqref{thm:R-bdd_dilations_FM;general_UMD;cond0} and \eqref{thm:R-bdd_dilations_FM;general_UMD;cond_infty} with the suprema taken over $\alpha \in \{0,1\}^{d}$ instead of $|\alpha| \leq d+2$;
as in the unweighted case, for $w \in A_{p}^{rec}(\R^{d})$ it can be shown that this classical Mihlin condition is sufficient for $m$ to be a Fourier multiplier on $L^{p}(\R^{d},w;X)$ (see \cite[Chapter~4]{Lindemulder_master-thesis}). In the unweighted case one could even use multiplier theorems which incorporate information of the Banach space under consideration \cite{Girardi&Weis_FM_geometry_BS,Hytonen_Fourier_embedding&Mihlin_type_multiplier_theorems}.
In Theorem~\ref{thm:R-bdd_dilations_via_Mihlin-Holder} (and Corollary~\ref{cor:thm:R-bdd_dilations_via_Mihlin-Holder}) we will actually use the Mihlin-H\"older condition from \cite[Theorem~3.1]{Hytonen_Fourier_embedding&Mihlin_type_multiplier_theorems} (which is weaker than the Mihlin-H\"ormander condition) for the one-dimensional case $d=1$.
\end{remark}

\begin{proof}
As in the proof of Corollary~\ref{cor:prop:R-bdd_partial_sums;dilations_constant_on_nbhd_origin}, it is enough to establish the $\mathcal{R}$-boundedness of $\{M_{j}\}_{j \in \N}$. Put $C:= \norm{M}_{\infty} \vee C_{0} \vee C_{\infty}$.
Pick $\zeta \in C^{\infty}_{c}(R^{d})$ with the property that $\chi(\xi)=1$ if $|\xi| \leq 1$ and $\zeta(\xi)=0$ if $|\xi| \geq 3/2$.
Then
\[
M  := M(0)\zeta + \zeta(M-M(0)\zeta) + (1-\zeta)(M-M(0)\zeta) =: M^{[1]} + M^{[2]} + M^{[3]}.
\]
For each $i \in \{1,2,3\}$ we define $\{M^{[i]}_{j}\}_{j \in \N}$ by $M^{[i]}_{j} := M^{i}(2^{-j}\,\cdot\,)$.
By Corollary~\ref{cor:prop:R-bdd_partial_sums;dilations_constant_on_nbhd_origin}, $\{M^{[1]}_{j}\}_{j \in \N}$ defines an $\mathcal{R}$-bounded sequence of Fourier multiplier operators in $\mathcal{B}(L^{p}(\R^{d},w;X))$ with $\mathcal{R}$-bound $\lesssim_{X,d,p,w,\zeta}|M(0)| \leq C$.
In order to get $\mathcal{R}$-boundedness for $i=2,3$ we use Corollary~\ref{prop:R-bdd_via_Mikhlin} (in combination with $\mathscr{M}_{d+2} \hookrightarrow \mathcal{M}_{p,w}(X)$).
To this end, let $\epsilon = \{\epsilon_{j}\}_{j=0}^{N} \in \{-1,1\}^{N+1}$, $N \in \N$, and put $M^{[i]}_{\epsilon} := \sum_{j=0}^{N}\epsilon_{j}M^{[i]}_{j}$ for each $i \in \{2,3\}$.
In order to obtain a uniform bound for $M^{[i]}_{\epsilon}$ in $\mathscr{M}_{d+2}$, we note that:
\begin{itemize}
\item $M^{[2]} \in C(\R^{d}) \cap C^{d+2}(\R^{d} \setminus \{0\})$ with $\supp M^{[2]} \subset B(0,2)$ and
\[
C^{[2]}:= \sup_{|\alpha| \leq d+2}\sup_{\xi \neq 0}|\xi|^{|\alpha|-\delta_{0}}|D^{\alpha}M^{[2]}(\xi)| \lesssim_{\zeta,\delta_{0},\delta_{\infty}} C;
\]
\item $M^{[3]} \in C^{d+2}(\R^{d})$ with $M^{[3]}(\xi)=0$ for $|\xi| \leq 1$ and
\[
C^{[3]}:= \sup_{|\alpha| \leq d+2}\sup_{\xi \neq 0}|\xi|^{|\alpha|+\delta_{\infty}}|D^{\alpha}M^{[2]}(\xi)| \lesssim_{\zeta,\delta_{0},\delta_{\infty}} C.
\]
\end{itemize}
For notational convenience, for each $j \geq N+1$ we write $\epsilon_{j}=0$.

\emph{The case $i=2$:} Let $|\alpha| \leq d+2$. For $\xi \in \bar{B}(0,2)$ we have
\begin{eqnarray*}
|\xi|^{|\alpha|}|D^{\alpha}M^{[2]}_{\epsilon}(\xi)|
&\leq& \sum_{j=0}^{\infty}|\xi|^{|\alpha|}|D^{\alpha}M^{[2]}_{j}(\xi)|
 = \sum_{j=0}^{\infty}|2^{-j}\xi|^{|\alpha|}|D^{\alpha}M^{[2]}(2^{-j}\xi)| \\
&\leq& C^{[2]}\sum_{j=0}^{\infty}|2^{-j}\xi|^{\delta_{0}} = C^{[2]}\left( \sum_{j=0}^{\infty}2^{-j\delta_{0}}\right)|\xi|^{\delta_{0}} \\
&\leq& C^{[2]}\frac{2^{\delta_{0}}}{1-2^{-\delta_{0}}}
\end{eqnarray*}
and for $\xi \in B(0,2^{l+1}) \setminus \bar{B}(0,2^{l})$, $l \in \N$, we similarly have, now using the support condition $\supp M^{[2]} \subset B(0,2)$,
\begin{eqnarray*}
|\xi|^{|\alpha|}|D^{\alpha}M^{[2]}_{\epsilon}(\xi)|
&\leq& \sum_{j=0}^{\infty}|\xi|^{|\alpha|}|D^{\alpha}M^{[2]}_{j}(\xi)|
 = \sum_{j=0}^{\infty}|2^{-j}\xi|^{|\alpha|}|D^{\alpha}M^{[2]}(2^{-j}\xi)| \\
&=& \sum_{j=l}^{\infty}|2^{-j}\xi|^{|\alpha|}|D^{\alpha}M^{[2]}_{j}(2^{-j}\xi)| \leq C^{[2]}\sum_{j=l}^{\infty}|2^{-j}\xi|^{\delta_{0}} \\
&=& C^{[2]}\left( \sum_{j=l}^{\infty}2^{-j\delta_{0}}\right)|\xi|^{\delta_{0}} \leq C^{[2]}\frac{2^{\delta_{0}}}{1-2^{-\delta_{0}}}.
\end{eqnarray*}
Hence, $\norm{M^{[2]}_{\epsilon}}_{\mathscr{M}_{d+2}} \leq C^{[2]}2^{\delta_{0}}(1-2^{-\delta_{0}})^{-1}$.

\emph{The case $i=3$:} Fix $l \in \N$. Since $M^{[3]} \equiv 0$ on $B(0,1)$, we have
\[
M^{[3]}_{\epsilon}(\xi) = \sum_{j=0}^{l}\epsilon_{j}M^{[3]}_{j}(\xi), \quad\quad \xi \in B(0,2^{l}) \setminus \bar{B}(0,2^{l-1}).
\]
For all $|\alpha| \leq d+2$ and $\xi \in B(0,2^{l}) \setminus \bar{B}(0,2^{l-1})$ we thus find
\begin{eqnarray*}
|\xi|^{|\alpha|}|D^{\alpha}M^{[3]}_{\epsilon}(\xi)|
&=& |\xi|^{|\alpha|}\Big|\sum_{j=0}^{l}\epsilon_{j}D^{\alpha}M^{[3]}_{j}(\xi)\Big|
   \leq \sum_{j=0}^{l}|\xi|^{|\alpha|} |D^{\alpha}M^{[3]}(\xi)| \\
&=& \sum_{j=0}^{l}|2^{-j}\xi|^{|\alpha|} |D^{\alpha}M^{[3]}(2^{-j}\xi)| \leq C^{[3]}\sum_{j=0}^{l}|2^{-j}\xi|^{-\delta_{\infty}}  \\
&\leq& C^{[3]}\sum_{j=0}^{l}(2^{-j+l-1})^{-\delta_{\infty}} = C^{[3]}2^{\delta_{\infty}}\sum_{j=0}^{l}2^{-\delta_{\infty}(l-j)} \\
&=& C^{[3]}2^{\delta_{\infty}}\sum_{k=0}^{l}2^{-\delta_{\infty} k} \leq  C^{[3]}\frac{2^{\delta_{\infty}}}{1-2^{-\delta_{\infty}}}.
\end{eqnarray*}
As $l \in \N$ was arbitrary and $M^{[3]}_{\epsilon} \equiv 0$ on $B(0,1)$, this shows that $\norm{M^{[3]}_{\epsilon}}_{\mathscr{M}_{d+2}} \leq  C^{[3]}2^{\delta_{\infty}}(1-2^{-\delta_{\infty}})^{-1}$.
\end{proof}

Note that Theorem~\ref{thm:R-bdd_dilations_FM;general_UMD} does not cover the symbol
$M(\xi) = \prod_{j=1}^{d}\mathrm{sinc}(\xi_{j})$, where $\mathrm{sinc}$ is the function given by $\mathrm{sinc}(t)=\frac{\sin(t)}{t}$ for $t \neq 0$ and $\mathrm{sinc}(0)=1$; see the end of Section~\ref{subsec:sec:difference_norms;main_result} for the relevance of this symbol, which is the Fourier transform of $2^{-d}1_{[-1,1]^{d}}$.
However, as already mentioned in Remark~\ref{rmk:thm:R-bdd_dilations_FM;general_UMD}, in the unweighted one-dimensional case we can use the Mihlin-H\"older multiplier theorem \cite[Theorem~3.1]{Hytonen_Mihlin-Hormander} in order to relax the conditions from Theorem~\ref{thm:R-bdd_dilations_FM;general_UMD}. This will lead to a criterium (Corollary~\ref{cor:thm:R-bdd_dilations_via_Mihlin-Holder}) which covers the symbol $M=\mathrm{sinc}$; see Example~\ref{ex:thm:difference_norms_separate_ineq_'sharp';R-bdd_indicator_cube}.

For each $k \in \Z$ and $j \in \{-1,1\}$ we define $I_{k,j} := j\,[2^{k-2},2^{k+2}]$.
For $\gamma \in (0,1)$ and $M \in C_{b}(\R \setminus \{0\})$ we put
\[
[M]_{\gamma}:= \sup_{k \in \Z, j = \pm 1}2^{k\gamma}[M|_{I_{k,j}}]_{C^{\gamma}(I_{k,j})}
\quad \mbox{and} \quad \normm{M}_{\gamma} := \norm{M}_{\infty} + [M]_{\gamma}.
\]
Since
\[
|M(\xi)-M(\xi-h)| \leq 4[M]_{\gamma}\,|h|^{\gamma}|\xi|^{-\gamma}, \quad |\xi| > 2|h|,
\]
the following lemma is a direct corollary of the vector-valued Mihlin-H\"older multiplier theorem \cite[Theorem~3.1]{Hytonen_Mihlin-Hormander}:
\begin{lemma}\label{lemma:mihlin-hormander}
Let $X$ be a UMD space and $p \in (1,\infty)$. Then there exists $\gamma_{X} \in (0,1)$, only depending on $X$, such that the following holds true: if $\gamma \in (\gamma_{X},1)$ and if $M \in C_{b}(\R \setminus \{0\})$ satisfies $\normm{M}_{\gamma} < \infty$, then $M$ defines a Fourier multiplier operator $T_{M}$ on $L^{p}(\R;X)$ of norm $\norm{T_{M}}_{\mathcal{B}(L^{p}(\R;X))} \lesssim_{X,p,\gamma}\normm{M}_{\gamma}$.\footnote{One can take $\gamma_{X}= \tau \vee q'$, where $\tau \in (1,2]$ and $q \in [2,\infty)$ denote the type and cotype of $X$, respectively. Here one needs the fact that $X$, as a UMD space, has non-trivial type and finite cotype; see \cite{boek}.}
\end{lemma}

Using this lemma, we find the following variant of Theorem~\ref{thm:R-bdd_dilations_FM;general_UMD}:

\begin{thm}\label{thm:R-bdd_dilations_via_Mihlin-Holder}
Let $X$ be a UMD space $p \in (1,\infty)$. Let $\gamma \in (\gamma_{X},1)$, where $\gamma_{X} \in (0,1)$ is from Lemma~\ref{lemma:mihlin-hormander}.
Let $M \in C_{b}(\R)$ and set $M_{n}:=M(2^{-n}\,\cdot\,)$ for each $n \in \Z$.
Suppose that there exist $\delta_{0},\delta_{\infty} > 0$ such that
\[
C_{0} := \sup_{0<|\xi|\leq 1}|\xi|^{-\delta_{0}}|M(\xi)-M(0)| \,\vee\, \sup_{k \leq -1,j=\pm1}2^{k(\gamma-\delta_{0})}[M|_{I_{k,j}}]_{C^{\gamma}(I_{k,j})} < \infty
\]
and
\[
C_{\infty} := \sup_{|\xi| \geq 1}|\xi|^{\delta_{\infty}}|M(\xi)| \,\vee\, \sup_{k \geq 0,j=\pm1}2^{k(\gamma+\delta_{\infty})}[M|_{I_{k,j}}]_{C^{\gamma}(I_{k,j})} < \infty.
\]
Then $\{M_{n}\}_{n \in \Z}$ defines an $\mathcal{R}$-bounded sequence of Fourier multiplier operators $\{T_{M_{n}}\}_{n \in \Z}$ in $\mathcal{B}(L^{p}(\R;X))$ with $\mathcal{R}$-bound $\lesssim_{X,p,\tau,q,\gamma,\delta_{0},\delta_{\infty}}[\norm{M}_{\infty} \vee C_{0} \vee C_{\infty}]$.
\end{thm}
\begin{proof}
This can be shown in a similar fashion as Theorem~\ref{thm:R-bdd_dilations_FM;general_UMD}, now using the (Mihlin-H\"older multiplier theorem in the form of) Lemma~\ref{lemma:mihlin-hormander} to treat the cases $i=2,3$.
\end{proof}

\begin{cor}\label{cor:thm:R-bdd_dilations_via_Mihlin-Holder}
Let $X$ be a UMD space $p \in (1,\infty)$. Let $\gamma \in (\gamma_{X},1)$, where $\gamma_{X} \in (0,1)$ is from Lemma~\ref{lemma:mihlin-hormander}. Let $M \in C_{b}(\R) \cap C^{1}(\R \setminus \{0\})$ and set $M_{n}:=M(2^{-n}\,\cdot\,)$ for each $n \in \Z$.
Suppose that there exist $\delta_{0},\delta_{\infty} > 0$ and $\theta \in [0,1]$ such that
\begin{equation}\label{eq:cor:thm:R-bdd_dilations_via_Mihlin-Holder;cond_origin}
C_{0} := \sup_{0<|\xi|\leq 1}|\xi|^{-\delta_{0}}|M(\xi)-M(0)| \,\vee\, \sup_{|\xi| \leq 1}|\xi|^{1-\delta_{0}}|M'(\xi)| < \infty
\end{equation}
and
\begin{equation}\label{eq:cor:thm:R-bdd_dilations_via_Mihlin-Holder;cond_infinity}
C_{\infty} := \sup_{|\xi| \geq 1}|\xi|^{\max\big\{\delta_{\infty},(\gamma+\delta_{\infty})\frac{1-\theta}{1-\gamma}\big\}}|M(\xi)|
\,\vee\, \sup_{|\xi| \geq 1}|\xi|^{(\gamma+\delta_{\infty})\frac{\theta}{\gamma}}|M'(\xi)|   < \infty.
\end{equation}
Then $\{M_{n}\}_{n \in \Z}$ defines an $\mathcal{R}$-bounded sequence of Fourier multiplier operators $\{T_{M_{n}}\}_{n \in \Z}$ in $\mathcal{B}(L^{p}(\R;X))$ with $\mathcal{R}$-bound $\lesssim_{X,p,\tau,q,\gamma,\delta_{0},\delta_{\infty}}[\norm{M}_{\infty} \vee C_{0} \vee C_{\infty}]$.
\end{cor}
\begin{proof}
For every $k \in \Z$ and $j \in \{-1,1\}$ we have
\[
2^{k(\gamma-\delta_{0})}[M|_{I_{k,j}}]_{C^{\gamma}(I_{k,j})} \lesssim_{\gamma} 2^{k(\gamma-\delta_{0})}2^{k(1-\gamma)}\norm{M'|_{I_{k,j}}}_{\infty}
\eqsim \sup_{\xi \in I_{k,j}}|\xi|^{1-\delta_{0}}|M'(\xi)|
\]
and
\begin{eqnarray*}
2^{k(\gamma+\delta_{\infty})}[M|_{I_{k,j}}]_{C^{\gamma}(I_{k,j})}
&\leq& 2^{k(\gamma+\delta_{\infty})}\,2^{1-\gamma}\norm{M|_{I_{k,j}}}_{\infty}^{1-\gamma}\norm{M'|_{I_{k,j}}}_{\infty}^{\gamma} \\
&\lesssim_{\gamma}&  2^{k(\gamma+\delta_{\infty})\frac{1-\theta}{1-\gamma}}\norm{M|_{I_{k,j}}}_{\infty} + 2^{k(\gamma+\delta_{\infty})\frac{\theta}{\gamma}}\norm{M'|_{I_{k,j}}}_{\infty} \\
&\eqsim& \sup_{\xi \in I_{k,j}}|\xi|^{(\gamma+\delta_{\infty})\frac{1-\theta}{1-\gamma}}|M(\xi)|
   + \sup_{\xi \in I_{k,j}}|\xi|^{(\gamma+\delta_{\infty})\frac{\theta}{\gamma}}|M'(\xi)|.
\end{eqnarray*}
The result now easily follows from Theorem~\ref{thm:R-bdd_dilations_via_Mihlin-Holder}.
\end{proof}

\section{Difference Norms}\label{sec:difference_norms}

\subsection{Notation}\label{subsec:sec:difference_norms;notation}

Let $X$ be a Banach space. For each $m \in \Z_{\geq 1}$ and $h \in \R^{d}$ we define difference operator $\Delta^{m}_{h}$ on $L^{0}(\R^{d};X)$ by
$\Delta^{m}_{h} := (L_{h}-I)^{m} = \sum_{j=0}^{m}(-1)^{j}{m \choose j}L_{(m-j)h}$, where $L_{h}$ denotes the left translation by $h$:
\[
\Delta^{m}_{h}f(x) = \sum_{j=0}^{m}(-1)^{j}{m \choose j}f(x+(m-j)h), \quad\quad f \in L^{0}(\R^{d};X), x \in \R^{d}.
\]

Let $p \in (1,\infty)$, $w \in A_{p}(\R^{d})$, $m \in \Z_{\geq 1}$, and $K \in \mathscr{K}(\R^{d})$.
For every $c>0$, $\tilde{K}_{c} = c^{d}K(-c\,\cdot\,) \in \mathscr{K}(\R^{d})$ gives rise to a (well-defined) bounded convolution operator $f \mapsto \tilde{K}_{c}*f$ on $L^{p}(\R^{d},w;X)$ of norm $\lesssim_{p,d,w} \norm{\tilde{K}_{c}}_{\mathscr{K}(\R^{d})} = \norm{K}_{\mathscr{K}(\R^{d})}$, which is given by the formula
\[
\tilde{K}_{c}*f\,(x) = \int_{\R^{d}}\tilde{K}_{c}(x-y)f(y)\,dy = \int_{\R^{d}}K(h)L_{c^{-1}h}f(x)\,dh, \quad\quad x \in \R^{d};
\]
see the last part of Section~\ref{subsec:prelim:weights}.
Defining $K^{\Delta^{m}} := \sum_{j=0}^{m-1}(-1)^{j}{m \choose j}\tilde{K}_{[(m-j)]^{-1}} \in \mathscr{K}(\R^{d})$,
for each $t>0$ the operator
\[
f \mapsto K_{m}(t,f) := K^{\Delta^{m}}_{t^{-1}}*f + (-1)^{m}\hat{K}(0)f = \sum_{j=0}^{m-1}(-1)^{j}{m \choose j}\tilde{K}_{[(m-j)t]^{-1}}*f + (-1)^{m}\hat{K}(0)f
\]
is bounded on $L^{p}(\R^{d},w;X)$ of norm $\lesssim_{p,d,w,m} \norm{K}_{\mathscr{K}(\R^{d})}$, and the following identity holds
\[
K_{m}(t,f)(x) = \int_{\R^{d}}K(h)\Delta^{m}_{th}f(x)\,dh, \quad\quad x \in \R^{d}.
\]
Given $f \in L^{p}(\R^{d},w;X)$,
the functions $K_{m}(t,f)$ may be interpreted as weighted means of differences of $f$.

For $f \in L^{p}(\R^{d},w;X)$ we set
\begin{equation*}
[f]_{H^{s}_{p}(\R^{d},w;X)}^{(m,K)} := \sup_{J \in \N}\normB{\sum_{j=1}^{J}\varepsilon_{j}2^{js}K_{m}(2^{-j},f)}_{L^{p}(\Omega;L^{p}(\R^{d},w;X))}.
\end{equation*}
and
\begin{equation*}
\normm{f}_{H^{s}_{p}(\R^{d},w;X)}^{(m,K)} := \norm{f}_{L^{p}(\R^{d},w;X)} + [f]_{H^{s}_{p}(\R^{d},w;X)}^{(m,K)}.
\end{equation*}

\subsection{Statement of the Main Result}\label{subsec:sec:difference_norms;main_result}

The following theorem is the main result of this paper. As already announced in the introduction, it is (indeed) a more general version of Theorem~\ref{thm:intro_main_result;difference} thanks to the $\mathcal{R}$-boundedness results Theorem~\ref{thm:R-bdd_dilations_FM;general_UMD} and Corollary~\ref{cor:thm:R-bdd_dilations_via_Mihlin-Holder}; see Examples \ref{ex:thm:difference_norms_separate_ineq_'sharp';examples_cond_rho} and \ref{ex:thm:difference_norms_separate_ineq_'sharp';R-bdd_indicator_cube}.

\begin{thm}\label{thm:difference_norms_separate_ineq_'sharp'}
Let $X$ be a UMD Banach space, $s > 0$, $p \in (1,\infty)$, $w \in A_{p}(\R^{d})$, $m \in \Z_{\geq 1}$ and $K \in \mathscr{K}(\R^{d})$.
\begin{itemize}
\item[(i)] Suppose that $K \in L^{1}(\R^{d},(1+|\,\cdot\,|)^{d+2})$ and that $K^{\Delta^{m}}$ fulfills the Tauberian condition
\begin{equation}\label{eq:thm:difference_norms_separate_ineq_'sharp';(i)_Tauberisch}
|\mathscr{F}K^{\Delta^{m}}(\xi)| \geq c, \quad\quad \xi \in \R^{d}, \frac{\epsilon}{2} < |\xi| < 2\epsilon,
\end{equation}
for some $\epsilon,c > 0$. Then we have the estimate
\begin{equation}\label{eq:thm:difference_norms_separate_ineq_'sharp';(i)_estimate}
\norm{f}_{H^{s}_{p}(\R^{d},w;X)} \lesssim \normm{f}_{H^{s}_{p}(\R^{d},w;X)}^{(m,K)},
\quad\quad f \in L^{p}(\R^{d},w;X).
\end{equation}
\item[(ii)] Suppose that $m > s$, $K \in L^{1}(\R^{d},(1+|\,\cdot\,|)^{(d+3)m})$, and that $\{ f \mapsto K_{m}(2^{-j},f) : j \in \Z_{\geq 1} \} \subset \mathcal{B}(L^{p}(\R^{d},w;X))$ is $\mathcal{R}$-bounded. Then we have the estimate
\begin{equation}\label{eq:thm:difference_norms_separate_ineq_'sharp';(ii)_estimate}
\normm{f}_{H^{s}_{p}(\R^{d},w;X)}^{(m,K)} \lesssim \norm{f}_{H^{s}_{p}(\R^{d},w;X)}, \quad\quad f \in L^{p}(\R^{d},w;X).
\end{equation}
\end{itemize}
\end{thm}

\begin{remark}
The $\mathcal{R}$-boundedness condition in (ii) of the above theorem may be replaced by the (at first sight) weaker condition that
\[
\normB{\sum_{j=1}^{N}\varepsilon_{j}K_{m}(2^{-j},g_{j})}_{L^{p}(\Omega;L^{p}(\R^{d},w;X))} \lesssim
\normB{\sum_{j=1}^{N}\varepsilon_{j}g_{j}}_{L^{p}(\Omega;L^{p}(\R^{d},w;X))}, \quad\quad N \in \N,
\]
for all $\{g_{j}\}_{j \geq 1} \subset L^{p}(\R^{d},w;X)$ with Fourier support $\supp \hat{g}_{j} \subset \{ \xi : |\xi| \geq c2^{j} \}$, where $c>0$ is some fixed number. But the $\mathcal{R}$-boundedness condition in (ii) is in fact implied by this condition.
Indeed, this condition implies the $\mathcal{R}$-boundedness of the sequence of Fourier multiplier operators associated with the the sequence of symbols $\big\{[(1-\zeta)\widehat{K^{\Delta^{m}}}](2^{-j}\,\cdot\,) \big\}_{j \geq 1}$, where $\zeta \in C^{\infty}_{c}(\R^{d})$ is a bump function which is $1$ on a neighborhood of the set $\{ \xi : |\xi| \geq c \}$. On the other hand, we have $\zeta\widehat{K^{\Delta^{m}}} \in C^{d+2}_{c}(\R^{d})$ in view of $\widehat{K^{\Delta^{m}}} \subset \mathscr{F}L^{1}(\R^{d},(1+|\,\cdot\,|)^{d+2}) \subset C_{b}^{d+2}(\R^{d})$, so that we can apply Theorem~\ref{thm:R-bdd_dilations_FM;general_UMD} to the symbol $\zeta\widehat{K^{\Delta^{m}}}$. We thus find that the sequence of symbols
$\big\{\widehat{K^{\Delta^{m}}_{2^{j}}} = \widehat{K^{\Delta^{m}}}(2^{-j}\,\cdot\,) \big\}_{j \geq 1}$ defines an $\mathcal{R}$-bounded sequence of Fourier multiplier operators on $L^{p}(\R^{d},w;X)$, which is of course equivalent to the $\mathcal{R}$-boundedness condition in (ii).
\end{remark}

\begin{remark}\label{rmk:thm:difference_norms_separate_ineq_'sharp';rmk_sum_integers}
Let $X$ be a Banach space, $s>0$, $p \in (1,\infty)$ and $w \in A_{p}(\R^{d})$. For each $f \in L^{p}(\R^{d},w;X)$ we put
\[
[f]_{H^{s}_{p}(\R^{d},w;X)}^{(m,K);\Z} := \sup_{J \in \N}\normB{\sum_{j=-J}^{J}\varepsilon_{j}2^{js}\varepsilon_{j}2^{js}K_{m}(2^{-j},f)\,}_{L^{p}(\Omega;L^{p}(\R^{d},w;X))}.
\]
On the one hand, $[\,\cdot\,]_{H^{s}_{p}(\R^{d},w;X)}^{(m,K)} \leq [\,\cdot\,]_{H^{s}_{p}(\R^{d},w;X)}^{(m,K);\Z}$ thanks to the contraction principle \eqref{eq:contraction_principle}. On the other hand, $[\,\cdot\,]_{H^{s}_{p}(\R^{d},w;X)}^{(m,K);\Z} \lesssim \norm{\,\cdot\,}_{L^{p}(\R^{d},w;X)} + [\,\cdot\,]_{H^{s}_{p}(\R^{d},w;X)}^{(m,K)}$ because $s>0$ and $\{ f \mapsto K_{m}(2^{-j},f) : j \in \Z \}$ is a uniformly bounded family in $\mathcal{B}(L^{p}(\R^{d},w;X))$.
In Theorem~\ref{thm:difference_norms_separate_ineq_'sharp'} we may thus replace $\normm{\,\cdot\,}_{H^{s}_{p}(\R^{d},w;X)}^{(m,K)}$ by $\norm{\,\cdot\,}_{L^{p}(\R^{d},w;X)} + [\,\cdot\,]_{H^{s}_{p}(\R^{d},w;X)}^{(m,K);\Z}$.
\end{remark}

\begin{ex}\label{ex:thm:difference_norms_separate_ineq_'sharp';examples_cond_rho}
Let $K \in \mathscr{K}(\R^{d})$ and $m \in \Z_{\geq 1}$.
\begin{itemize}
\item[(i)]  Note that $\mathscr{F}K^{\Delta^{m}} \in C_{b}(\R^{d})$ with $\mathscr{F}K^{\Delta^{m}}(0) = \sum_{j=0}^{m-1}(-1)^{j}{m \choose j}\hat{K}(0) = (-1)^{m+1}\hat{K}(0)$. So for $K^{\Delta^{m}}$ to fulfill the Tauberian condition \eqref{eq:thm:difference_norms_separate_ineq_'sharp';(i)_Tauberisch} for some $\epsilon,c > 0$ it is sufficient that $\hat{K}(0) \neq 0$.
\item[(ii)] Let $X$ be a UMD space, $p \in (1,\infty)$ and $w \in A_{p}(\R^{d})$. Note that the $\mathcal{R}$-boundedness conditition in Theorem~\ref{thm:difference_norms_separate_ineq_'sharp'}(ii) is equivalent to the $\mathcal{R}$-boundedness of the convolution operators $\{f \mapsto K^{\Delta^{m}}_{2^{j}}*f : j \in \Z_{\geq 1}  \} \subset \mathcal{B}(L^{p}(\R^{d},w;X))$. By Theorem~\ref{thm:R-bdd_dilations_FM;general_UMD}, for the latter it is sufficient that $K \in L^{1}(\R^{d},(1+|\,\cdot\,|)^{d+2}) \subset \mathscr{F}^{-1}C^{d+2}_{b}(\R^{d})$ fulfills the condition
    \begin{equation}\label{eq:ex:thm:difference_norms_separate_ineq_'sharp';examples_cond_rho;(ii)}
    \sup_{|\alpha| \leq d+2}\sup_{\xi \in \R^{d}}(1+|\xi|)^{|\alpha|+\delta}|D^{\alpha}\hat{K}(\xi)| < \infty
    \end{equation}
    for some $\delta > 0$; in particular, it is sufficient that $K \in \mathcal{S}(\R^{d})$.
\end{itemize}
\end{ex}

Under the availability of better multiplier theorems than $\mathscr{M}_{d+2}(\R^{d}) \hookrightarrow \mathcal{M}_{p,w}(X)$, the condition \eqref{eq:ex:thm:difference_norms_separate_ineq_'sharp';examples_cond_rho;(ii)} can be weakened; see Remark~\ref{rmk:thm:R-bdd_dilations_FM;general_UMD}. For example, in the one-dimensional case $d=1$ we can use $\mathscr{M}_{1}(\R) \hookrightarrow \mathcal{M}_{p,w}(X)$, resulting in the weaker condition that
\[
\sup_{k=0,1}(1+|\xi|)^{k+\delta}|\hat{K}^{(k)}(\xi)| < \infty
\]
for some $\delta > 0$. However, this condition is still to strong to handle the kernel $K=2^{-1}1_{[-1,1]} \in L^{\infty}_{c}(\R^{d}) \subset \mathscr{K}(\R^{d}) \cap \mathscr{F}^{-1}C^{\infty}_{0}(\R^{d})$ with Fourier transform $\hat{K} = \mathrm{sinc}$, where $\mathrm{sinc}(t)=\sin(t)/t$ for $t \neq 0$ and $\mathrm{sinc}(0)=1$. As already announced, in the unweighted case this $K$ can be handled by Corollary~\ref{cor:thm:R-bdd_dilations_via_Mihlin-Holder}:
\begin{ex}\label{ex:thm:difference_norms_separate_ineq_'sharp';R-bdd_indicator_cube}
Let $X$ be a UMD Banach space, $p \in (1,\infty)$ and $K=2^{-d}1_{Q[0,1]}$.
For every $m \in \Z_{\geq 1}$ it holds that $\{ f \mapsto K_{m}(2^{-j},f) : j \in \Z \} \subset \mathcal{B}(L^{p}(\R^{d};X))$ is $\mathcal{R}$-bounded.
\end{ex}
\begin{proof}
It is enough to show that $\{T_{\hat{K}(\ell2^{-j}\,\cdot\,)} : j \in \Z, \ell \in \{1,\ldots,m\} \} = \{ f \mapsto K_{\ell^{-1}2^{j}}*f : j \in \Z, \ell \in \{1,\ldots,m\} \}$ is $\mathcal{R}$-bounded in $\mathcal{B}(L^{p}(\R^{d};X))$. By the product structure of $K$ it suffices to consider the case $d=1$. So we only need to check that $M := \mathrm{sinc} = \mathscr{F}\frac{1}{2}1_{[-1,1]} \in C^{\infty}_{0}(\R)$ satisfies the conditions from Corollary~\ref{cor:thm:R-bdd_dilations_via_Mihlin-Holder}. In the notation of Corollary~\ref{cor:thm:R-bdd_dilations_via_Mihlin-Holder}, let $\gamma \in (\gamma_{X},1)$ be fixed. The condition \eqref{eq:cor:thm:R-bdd_dilations_via_Mihlin-Holder;cond_origin} is fulfilled for $\delta_{0}=1$ because $\mathrm{sinc}$ is a $C^{1}$-function on $[-1,1]$. Furthermore, the condition \eqref{eq:cor:thm:R-bdd_dilations_via_Mihlin-Holder;cond_infinity} is fulfilled for any $\delta_{\infty} \in (0,1-\gamma)$ and $\theta=\gamma$.
\end{proof}

Still consider $K=2^{-d}1_{Q[0,1]} \in L^{\infty}_{c}(\R^{d}) \subset \mathscr{K}(\R^{d})$.
The $\mathcal{R}$-boundedness condition from Theorem~\ref{thm:difference_norms_separate_ineq_'sharp'}(ii) is fulfilled provided that, for each $\ell \in \{1,\ldots,m\}$, the set of convolution operators $\{f \mapsto K_{t}*f : t=\ell^{-1}2^{j}, j \in \Z_{\geq 1}  \} \subset \mathcal{B}(L^{p}(\R^{d},w;X))$ is $\mathcal{R}$-bounded. A nice way to look at the convolution operator $f \mapsto K_{r^{-1}}*f$, $r>0$, is as the averaging operator $A_{r} \in \mathcal{B}(L^{p}(\R^{d},w;X))$ given by
\[
A_{r}f(x) := \fint_{Q[x,r]}f(y)\,dy, \quad\quad f \in L^{p}(\R^{d},w;X), x \in \R^{d}.
\]
This leads to the following natural question:
\begin{question}\label{question:R-bdd_averaging_operators}
Given a UMD space $X$, $p \in (1,\infty)$, $w \in A_{p}(\R^{d})$ and $c > 0$, is the set of averaging operators $\{A_{r} : r = c2^{-j},j \in \Z_{\geq 1}  \}$  $\mathcal{R}$-bounded in $\mathcal{B}(L^{p}(\R^{d},w;X))$?
\end{question}

Three cases in which we can give a positive answer to this question are:
\begin{itemize}
\item[(i)] $X$ is a UMD space, $p \in (1,\infty)$ and $w=1$;
\item[(ii)] $X$ is a UMD space with property $(\alpha)$, $p \in (1,\infty)$ and $w \in A_{p}^{rec}(\R^{d})$;\footnote{Recall that $A_{p}^{rec}$ is the class of weights on $\R^{d}$ which are uniformly $A_{p}$ in each of the coordinates separately.}
\item[(iii)] $X$ is a UMD Banach function space, $p \in (1,\infty)$ and $w \in A_{p}(\R^{d})$;
\end{itemize}
Here case (i) follows similarly to the proof of Example~\ref{ex:thm:difference_norms_separate_ineq_'sharp';R-bdd_indicator_cube}, case (ii) follows from an application of Proposition~\ref{prop:R-bdd_FM_alpha}(ii), and case (iii) can be treated via the Banach lattice version of the Hardy-Littlewood maximal function by using the fact that $\mathcal{R}$-boundedness coincides with $\ell^{2}$-boundedness in this situation (see Proposition~\ref{prop:R-bdd_convolutions_UMD_BFS} for a more general result in this direction).
Note that in the cases (ii) and (iii) one in fact has $\mathcal{R}$-boundedness of $\{A_{r}:r>0\}$ in $\mathcal{B}(L^{p}(\R^{d},w;X))$

\subsection{Proof of the Main Result}

Below we will use the following notation:
\begin{equation*}
X_{p,w} := L^{p}(\Omega;L^{p}(\R^{d},w;X)) = L^{p}(\R^{d},w;L^{p}(\Omega;X)).
\end{equation*}
\begin{equation*}
X_{p,w}(\R^{d}_{\pm}) := L^{p}(\Omega;L^{p}(\R^{d}_{\pm},w;X)) = L^{p}(\R^{d}_{\pm},w;L^{p}(\Omega;X)).
\end{equation*}

\paragraph{\textbf{Proof of Theorem~\ref{thm:difference_norms_separate_ineq_'sharp'}(i).}}

\begin{lemma}\label{lemma:afsch_Bessel_potential_by_convolutions}
Let $X$ be a UMD space, $s \in \R$, $p \in (1,\infty)$ and $w \in A_{p}(\R^{d})$.
Suppose that $k \in \mathscr{K}(\R^{d}) \cap L^{1}(\R^{d},(1+|\,\cdot\,|)^{d+2})$ fulfills the Tauberian condition
\[
|\hat{k}(\xi)| > 0, \quad\quad \xi \in \R^{d}, \frac{\epsilon}{2} < |\xi| < 2\epsilon,
\]
for some $\epsilon > 0$.
For $f \in L^{p}(\R^{d},w;X)$ we can then estimate
\begin{equation}\label{eq:lemma:afsch_Bessel_potential_by_convolutions;estimate}
\norm{f}_{H^{s}_{p}(\R^{d},w;X)} \lesssim \norm{f}_{L^{p}(\R^{d},w;X)} +
\sup_{J \in \N}\normB{\sum_{j=1}^{J}\varepsilon_{j}2^{js}k_{j}*f}_{X_{p,w}}.
\end{equation}
\end{lemma}
\begin{proof}
Pick $\varphi=(\varphi_{j})_{j \in \N} \in \Phi(\R^{d})$ such that $\supp \hat{\varphi}_{1} \subset \{ \xi : |\xi| \geq 2\epsilon \}$; see \eqref{eq:Fourier_support_LP-seq}. Using \eqref{eq:randomized-LP-decomp} in combination with $S_{0} \in \mathcal{B}(L^{p}(\R^{d},w;X))$, we get
\[
\norm{f}_{H^{s}_{p}(\R^{d},w;X)} \lesssim \norm{f}_{L^{p}(\R^{d},w;X)} + \sup_{J \in \N}\normB{\sum_{j=1}^{J}\varepsilon_{j}2^{js}S_{j}f}_{X_{p,w}}.
\]
In view of the contraction principle \eqref{eq:contraction_principle}, it is thus enough to find an $N \in \N$ such that
\begin{equation}\label{eq:lemma:afsch_Bessel_potential_by_convolutions;suffices}
\normB{\sum_{j=1}^{J}\varepsilon_{j}2^{js}S_{j}f}_{X_{p,w}} \lesssim
\normB{\sum_{j=1}^{J+N}\varepsilon_{j}2^{js}k_{j}*f}_{X_{p,w}}, \quad\quad f \in L^{p}(\R^{d},w;X), J \in \N.
\end{equation}

In order to establish \eqref{eq:lemma:afsch_Bessel_potential_by_convolutions;suffices},
pick $\eta \in C^{\infty}_{c}(\R^{d})$ with $\supp \eta \subset B(0,2\epsilon)$ and $\eta(\xi)=1$ for $|\xi| \leq~\frac{3\epsilon}{2}$.
Define $m \in C^{d+2}_{c}(\R^{d}) \subset \mathscr{M}_{d+2}(\R^{d})$ by $m(\xi):= [\eta(\xi)-\eta(2\xi)]\hat{k}(\xi)^{-1}$ if $\frac{\epsilon}{2} < |\xi| < 2\epsilon$ and $m(\xi):=0$ otherwise; note that this gives a well-defined $C^{d+2}$-function on $\R^{d}$ because $\eta-\eta(2\,\cdot\,)$ is a smooth function supported in the set $\{ \xi : \frac{\epsilon}{2} < |\xi| < 2\epsilon\}$ on which the function $\hat{k} \in C^{d+2}(\R^{d})$ does not vanish, where the regularity $\hat{k} \in C^{d+2}(\R^{d})$ is a consequence of the assumption that $k \in L^{1}(\R^{d},(1+|\,\cdot\,|)^{d+2})$.
By Example~\ref{ex:prop;R-bddness_fm;nbhd_cond_non-zero_restr;ex_cond_(a)en(b)}, the sequence of (dyadic) dilated symbols $\{m_{j}:=m(2^{-j}\,\cdot\,)\}_{j \geq 1}$ defines an $\mathcal{R}$-bounded sequence of Fourier multiplier operators $\{T_{m_{j}}\}_{j \geq 1}$ on $L^{p}(\R^{d},w;X)$.
Furthermore, by construction we have
\[
\sum_{l=j}^{j+N}m_{l}\hat{k}_{l}(\xi) = \eta(2^{-(j+N)}\xi)-\eta(2^{-j+1}\xi) = 1 \quad
\mbox{for}\:\: 2^{j}\epsilon \leq |\xi| \leq 2^{j+N-1}3\epsilon, j \geq 1, N \in \N.
\]
Since $\supp \hat{\varphi}_{j} \subset \{ \xi : 2^{j}\epsilon \leq |\xi| < 2^{j}B \}$ for every $j \geq 1$ for some $B>\epsilon$, there thus exists $N \in \N$ such that $\sum_{l=j}^{j+N}m_{l}\hat{k}_{l} \equiv 1$ on $\supp \hat{\varphi}_{j}$ for all $j \geq 1$.
For each $j \geq 1$ we consequently have
\[
S_{j} = T_{\hat{\varphi}_{j}} = T_{\hat{\varphi}_{j}\left(\sum_{l=j}^{j+N}m_{l}\hat{k}_{l}\right)} =
\sum_{l=j}^{j+N}T_{\hat{\varphi}_{j}}T_{m_{l}}T_{\hat{k}_{l}} = \sum_{l=0}^{N}S_{j}T_{m_{j+l}}[k_{j+l}*\,\cdot\,]
\quad \mbox{in}\:\: \mathcal{B}(L^{p}(\R^{d},w;X)).
\]
Using this together with the $\mathcal{R}$-boundedness of $\{S_{j}\}_{j \in \N}$ and $\{T_{m_{j}}\}_{j \geq 1}$ (see Example~\ref{ex:prop;R-bddness_fm;nbhd_cond_non-zero_restr;ex_cond_(a)en(b)}), for each $f \in L^{p}(\R^{d},w;X)$ we obtain the estimates
\begin{eqnarray*}
\normB{\sum_{j=1}^{J}\varepsilon_{j}2^{js}S_{j}f}_{X_{p,w}}
&\leq& \sum_{l=0}^{N}\normB{\sum_{j=1}^{J}\varepsilon_{j}2^{js}S_{j}T_{m_{j+l}}[k_{j+l}*f]}_{X_{p,w}} \\
&\lesssim& \sum_{l=0}^{N}\normB{\sum_{j=1}^{J}\varepsilon_{j}2^{js}k_{j+l}*f}_{X_{p,w}} \\
&\lesssim& \normB{\sum_{j=1}^{J+N}\varepsilon_{j}2^{js}k_{j}*f}_{X_{p,w}}.
\end{eqnarray*}
\end{proof}

\begin{proof}[Proof of Theorem~\ref{thm:difference_norms_separate_ineq_'sharp'}(i)]
In view of \eqref{eq:thm:difference_norms_separate_ineq_'sharp';(i)_Tauberisch} and the fact that $\mathscr{F}K^{\Delta^{m}} \in C_{0}(\R^{d})$, there exists $N \in \N$ such that the function $k \in \mathscr{K}(\R^{d}) \cap L^{1}(\R^{d},(1+|\,\cdot\,|)^{d+2})$ determined by  $\hat{k} = \mathscr{F}K^{\Delta^{m}}(2^{-N}\,\cdot\,) - \mathscr{F}K^{\Delta^{m}}$ fulfills the Tauberian condition
\[
|\hat{k}(\xi)| \geq \frac{c}{2} > 0, \quad\quad \xi \in \R^{d}, \frac{\delta}{2} < |\xi| < 2\delta,
\]
for $\delta:= 2^{N}\epsilon > 0$. Since
\begin{eqnarray*}
k_{j}*f
&=& [K^{\Delta^{m}}_{2^{-(j+N)}}*f +(-1)^{m}\hat{K}(0)f] - [K^{\Delta^{m}}_{2^{-j}}*f +(-1)^{m}\hat{K}(0)f] \\
&=& K_{m}(2^{-(j+N)},f) - K_{m}(2^{-j},f), \quad\quad j \geq 1,
\end{eqnarray*}
with Lemma~\ref{lemma:afsch_Bessel_potential_by_convolutions} it follows that
\begin{eqnarray*}
\norm{f}_{H^{s}_{p}(\R^{d},w;X)}
&\lesssim& \norm{f}_{L^{p}(\R^{d},w;X)} + \sup_{J}\normB{\sum_{j=1}^{J}\varepsilon_{j}2^{js}k_{j}*f}_{X_{p,w}} \\
&\lesssim& \norm{f}_{L^{p}(\R^{d},w;X)} +
\sup_{J}2^{-Ns}\normB{\sum_{j=1}^{J}\varepsilon_{j}2^{(j+N)s}K_{m}(2^{-(j+N)},f)}_{X_{p,w}} \\
&& \quad +\: \sup_{J}\normB{\sum_{j=1}^{J}\varepsilon_{j}2^{js}K_{m}(2^{-j},f)}_{X_{p,w}}\\
&\stackrel{\eqref{eq:contraction_principle}}{\leq}& \norm{f}_{L^{p}(\R^{d},w;X)} + (2^{-Ns}+1)[f]_{H^{s}_{p}(\R^{d},w;X)}^{(m,K)}.
\end{eqnarray*}

\end{proof}

\paragraph{\textbf{Proof of Theorem~\ref{thm:difference_norms_separate_ineq_'sharp'}(ii).}}

\begin{lemma}\label{lemma:fm_differences_with_support_cond}
Let $X$ be a UMD space, $p \in (1,\infty)$ and $w \in A_{p}(\R^{d})$.
Let $\chi \in C^{\infty}_{c}(\R^{d}\setminus\{0\})$ and $\eta \in C^{\infty}_{c}(\R^{d})$.
For each $n \in \Z_{\leq 0}$ and $h \in \R^{d}$ we define the sequence of symbols
$\{M^{h,n}_{j}\}_{j \in \Z} \subset L^{\infty}(\R^{d})$ by
\[
M^{h,n}_{j}(\xi) := \left\{\begin{array}{ll}
(e^{\imath 2^{-j}h\cdot\xi}-1)\chi(2^{-(n+j)}\xi), & n+j \geq 1\\
(e^{\imath 2^{-j}h\cdot\xi}-1)\eta(2^{-(n+j)}\xi), & n+j = 0 \\
0, & n+j \leq -1
\end{array}\right.
\]
Then each symbol $M^{h,n}_{j}$ defines a bounded Fourier multiplier operator $T^{h,n}_{j} = T_{M^{h,n}_{j}}$ on $L^{p}(\R^{d},w;X)$ such that the following $\mathcal{R}$-bound is valid:
\begin{equation}\label{eq:lemma:fm_differences_with_support_cond;R-bound}
\mathcal{R}\{ T^{h,n}_{j} : j \in \Z\} \lesssim 2^{n}(1+|h|)^{d+3}, \quad\quad h \in \R^{d}, n \in \Z_{\leq 0}.
\end{equation}
\end{lemma}
\begin{proof}
By construction, $\{M^{h,n}_{j}\}_{j \in \Z} \subset C^{\infty}_{c}(\R^{d})$ satisfies condition (a) of Corollary~\ref{prop;R-bddness_fm;nbhd_cond_non-zero_restr} for some $N \in \N$ independent of $n \in \Z_{\leq 0}$ and $h \in \R^{d}$.
Therefore, it is enough to show that
\begin{equation}\label{eq:lemma:fm_differences_with_support_cond;suffices}
\norm{M^{h,n}_{j}}_{\mathscr{M}_{d+2}} \lesssim 2^{n}(1+|h|)^{d+3}, \quad\quad h \in \R^{d}, n \in \Z_{\leq 0}, j \in \Z.
\end{equation}

We only consider the case $n+j \geq 1$ in \eqref{eq:lemma:fm_differences_with_support_cond;suffices}, the case $n+j=0$ being comletely similar and the case $n+j \leq -1$ being trivial. Let $h \in \R^{d}$, $n \in \Z_{\leq 0}$ and $j \in \Z$ with $n+j \geq 1$ be given.
Fix a multi-index $\alpha \in \N^{d}$ with $|\alpha| \leq d+2$.
Using the Leibniz rule, we compute
\begin{eqnarray*}
|\xi|^{|\alpha|}D^{\alpha}M^{h,n}_{j}(\xi)
&=& |\xi|^{|\alpha|}D^{\alpha}_{\xi}\left( \imath h \cdot \xi \int_{0}^{2^{-j}}e^{\imath sh \cdot \xi}ds \, \chi(2^{-(n+j)}\xi) \right) \\
&=& \imath \sum_{\beta+\gamma \leq \alpha}c^{\alpha}_{\beta,\gamma} |\xi|^{|\beta|}D^{\beta}_{\xi}(h \cdot \xi)\,|\xi|^{|\gamma|}D^{\gamma}_{\xi}\left( \int_{0}^{2^{-j}}e^{\imath sh \cdot \xi}ds\right) \,|\xi|^{|\alpha|-|\beta|-|\gamma|}D^{\alpha-\beta-\gamma}_{\xi}[\chi(2^{-(n+j)}\xi)] \\
&=& \imath \sum_{\gamma \leq \alpha}c^{\alpha}_{0,\gamma} \, h \cdot \xi \,\,|\xi|^{|\gamma|} \int_{0}^{2^{-j}}(\imath sh)^{\gamma}e^{\imath sh \cdot \xi}ds\, |2^{-(n+j)}\xi|^{|\alpha|-|\gamma|}[D^{\alpha-\gamma}\chi](2^{-(n+j)}\xi) \\
&& + \:\: \imath \sum_{\beta+\gamma \leq \alpha; |\beta| = 1}c^{\alpha}_{\beta,\gamma} |\xi|h^{\beta} \,|\xi|^{|\gamma|}\int_{0}^{2^{-j}}(\imath sh)^{\gamma}e^{\imath sh \cdot \xi}ds \, |2^{-(n+j)}\xi|^{|\alpha|-|\beta|-|\gamma|}[D^{\alpha-\beta-\gamma}\chi](2^{-(n+j)}\xi).
\end{eqnarray*}
Picking $R > 0$ such that $\supp \chi \subset B(0,R)$, we can estimate
\begin{eqnarray*}
|\xi|^{|\alpha|}|D^{\alpha}M^{h,n}_{j}(\xi)|
&\lesssim& \sum_{\gamma \leq \alpha}  \,|h|^{|\gamma|+1}2^{-j(|\gamma|+1)}\,1_{B(0,R)}(2^{-(n+j)}\xi)\,|\xi|^{|\gamma|+1}\,\norm{\chi}_{\mathscr{M}_{d+2}} \\
&& + \:\: \sum_{\beta+\gamma \leq \alpha; |\beta| = 1} \,|h|^{|\gamma|+1}2^{-j(|\gamma|+1)}\,1_{B(0,R)}(2^{-(n+j)}\xi)\,|\xi|^{|\gamma|+1}\,\norm{\chi}_{\mathscr{M}_{d+2}} \\
&\leq& 2\norm{\chi}_{\mathscr{M}_{d+2}} \sum_{\gamma \leq \alpha}|h|^{|\gamma|+1}2^{n(|\gamma|+1)}R^{|\gamma|+1} \\
&\stackrel{n \leq 0}{\lesssim}& 2^{n}(1+|h|)^{d+3}.
\end{eqnarray*}
This proves the required estimate \eqref{eq:lemma:fm_differences_with_support_cond;suffices}.
\end{proof}

\begin{proof}[Proof of Theorem~\ref{thm:difference_norms_separate_ineq_'sharp'}(ii)]
Given $f \in L^{p}(\R^{d},w;X)$, write $f_{n}:=S_{n}f$ for $n \in \N$ and $f_{n}:=0$ for $n \in \Z_{<0}$.
For each $j \in \Z_{>0}$ we then have $f = \sum_{n \in \Z}f_{n+j}$ in $L^{p}(\R^{d},w;X)$, from which it follows that
\begin{equation}\label{eq:thm:difference_norms_separate_ineq_'sharp';(ii)_sum_n}
\normB{\sum_{j=1}^{J}\varepsilon_{j}2^{js}K_{m}(2^{-j},f)}_{X_{p,w}}
\leq \sum_{n \in \Z}\normB{\sum_{j=1}^{J}\varepsilon_{j}2^{js}K_{m}(2^{-j},f_{n+j})}_{X_{p,w}}.
\end{equation}

We first estimate the sum over $n \in \Z_{>0}$ in \eqref{eq:thm:difference_norms_separate_ineq_'sharp';(ii)_sum_n}.
Using the $\mathcal{R}$-boundedness of $\{f \mapsto K_{m}(2^{-j},f) : j \geq 1 \}$, we find
\[
\normB{\sum_{j=1}^{J}\varepsilon_{j}2^{js}K_{m}(2^{-j},f_{n+j})}_{X_{p,w}} \lesssim
2^{-ns}\normB{\sum_{j=1}^{J}\varepsilon_{j}2^{(n+j)s}f_{n+j}}_{X_{p,w}} \leq 2^{-ns}\norm{f}_{H^{s}_{p}(\R^{d},w;X)}.
\]
Since $s>0$, it follows that the sum over $n \in \Z_{>0}$ in \eqref{eq:thm:difference_norms_separate_ineq_'sharp';(ii)_sum_n} can be estimated from above by $C\norm{f}_{H^{s}_{p}(\R^{d},w;X)}$ for some constant $C$ independent of $f$ and $J$.

Next we estimate the sum over $n \in \Z_{\leq 0}$ in \eqref{eq:thm:difference_norms_separate_ineq_'sharp';(ii)_sum_n}.
To this end, let $\chi \in C^{\infty}_{c}(\R^{d} \setminus \{0\})$ and $\eta \in C^{\infty}_{c}$ be such that $\chi \equiv 1$ on $\frac{1}{2}\supp \hat{\varphi}_{1}$ and $\eta \equiv 1$ on $\supp \hat{\varphi}_{0}$.
For every $\lambda \in \C$ we define the function $e_{\lambda}:\R^{d} \to \C$ by $e_{\lambda}(\xi):= e^{\lambda \cdot \xi}$.
For each $n \leq 0$, $h \in \R^{d}$ and $j \geq 1$, we then have
\begin{eqnarray*}
\Delta^{m}_{2^{-j}h}f_{n+j}
&=& \mathscr{F}^{-1}[(e_{\imath 2^{-j}h}-1)^{m}\hat{f}_{n+j}] \\
&=& \left\{\begin{array}{ll}
\mathscr{F}^{-1}\left[\left(e_{\imath 2^{-j}h}-1)\chi(2^{-(n+j)}\,\cdot\,)\right)^{m}\hat{f}_{n+j}\right], & n+j \geq 1;\\
\mathscr{F}^{-1}\left[\left(e_{\imath 2^{-j}h}-1)\eta(2^{-(n+j)}\,\cdot\,)\right)^{m}\hat{f}_{n+j}\right], & n+j = 0; \\
0, & n+j \leq -1.
\end{array}\right. \\
&=& T_{M^{h,n}_{j}}^{m}f_{n+j},
\end{eqnarray*}
where $M^{h,n}_{j}$ is the Fourier multiplier symbol from Lemma~\ref{lemma:fm_differences_with_support_cond}.
For each $n \leq 0$ we thus get
\begin{eqnarray*}
\normB{\sum_{j=1}^{J}\varepsilon_{j}2^{js}K_{m}(2^{-j},f_{n+j})}_{X_{p,w}}
&\leq& \int_{\R^{d}}|K(h)|\,\normB{\sum_{j=1}^{J}\varepsilon_{j}2^{js}\Delta^{m}_{2^{-j}h}f_{n+j}(\,\cdot\,)}_{X_{p,w}}dh \\
&=& \int_{\R^{d}}|K(h)|\,\normB{\sum_{j=1}^{J}\varepsilon_{j}2^{js}T_{M^{h,n}_{j}}^{m}f_{n+j}}_{X_{p,w}}dh \\
&\stackrel{\eqref{eq:lemma:fm_differences_with_support_cond;R-bound}}{\lesssim}&
2^{n(m-s)}\int_{\R^{d}}|K(h)|(1+|h|)^{(d+3)m}dh  \\
&& \quad\quad \cdot \:\: \normB{\sum_{j=1}^{J}\varepsilon_{j}2^{(n+j)s}f_{n+j}}_{X_{p,w}} \\
&\stackrel{\eqref{eq:randomized-LP-decomp}}{\lesssim}& 2^{n(m-s)}\norm{f}_{H^{s}_{p}(\R^{d},w;X)}.
\end{eqnarray*}
Since $m-s > 0$, it follows that the sum over $n \in \Z_{\leq 0}$ in \eqref{eq:thm:difference_norms_separate_ineq_'sharp';(ii)_sum_n} can be estimated from above by $C\norm{f}_{H^{s}_{p}(\R^{d},w;X)}$ for some constant $C$ independent of $f$ and $J$.
\end{proof}

The idea to do the estimate \eqref{eq:thm:difference_norms_separate_ineq_'sharp';(ii)_sum_n} and to treat the sum over $n \in \Z_{>0}$ and $n \in \Z_{\leq 0}$ separately is taken from the proof of \cite[Proposition~6]{S&S_jena-notes}, which is concerned with a difference norm characterization for $F^{s}_{p,q}(\R^{d};X)$.

\subsection{The Special Case of a Banach Function Space}\label{sec:BFS-case}

In the special case that $X$ is a Banach function space, we obtain the following corollary from the main result Theorem~\ref{thm:difference_norms_separate_ineq_'sharp'}:

\begin{cor}\label{cor:difference_norms_separate_ineq_'sharp';BFS_equiv_norm_differences}
Let $X$ be a UMD Banach function space, $s>0$, $p \in (1,\infty)$, $w \in A_{p}(\R^{d})$ and $m \in \N$, $m >s$.
Suppose that $K \in  \mathscr{K}(\R^{d}) \cap L^{1}(\R^{d},(1+|\,\cdot\,|)^{(d+3)m})$ satisfies the Tauberian condition \eqref{eq:thm:difference_norms_separate_ineq_'sharp';(i)_Tauberisch} for some $c,\epsilon > 0$. For all $f \in L^{p}(\R^{d},w;X)$ we then have the equivalence of extended norms
\begin{equation}\label{eq:cor:difference_norms_separate_ineq_'sharp';BFS_equiv_norm_differences}
\norm{f}_{H^{s}_{p}(\R^{d},w;X)} \eqsim \norm{f}_{L^{p}(\R^{d},w;X)} + \normB{ \Big(\,\sum_{j=1}^{\infty}|2^{js}K_{m}(2^{-j},f)|^{2}\,\Big)^{1/2} }_{L^{p}(\R^{d},w;X)}.
\end{equation}
\end{cor}
\begin{proof}
By the Khintchine-Maurey theorem, the right-hand side (RHS) of \eqref{eq:cor:difference_norms_separate_ineq_'sharp';BFS_equiv_norm_differences} defines an extended norm on $L^{p}(\R^{d},w;X)$ which is equivalent to $\normm{\,\cdot\,}^{(m,K)}_{H^{s}_{p}(\R^{d},w;X)}$. Therefore, we only need to check the $\mathcal{R}$-boundedness condition in Theorem~\ref{thm:difference_norms_separate_ineq_'sharp'}(ii). But this follows from Proposition~\ref{prop:R-bdd_convolutions_UMD_BFS} below (and the discussion after it).
\end{proof}

\begin{remark}
Let $X$ be a UMD Banach function space, $s>0$, $p \in (1,\infty)$, $w \in A_{p}(\R^{d})$ and $m \in \N$, $m >s$.
Suppose $K \in \mathscr{K}(\R^{d})^{+} \setminus \{0\}$. Then it is a natural question whether we can replace $K_{m}(2^{-j},f)$ by
$d^{m}_{K}(2^{-j},f)$ in the RHS of \eqref{eq:cor:difference_norms_separate_ineq_'sharp';BFS_equiv_norm_differences}, where
\[
d^{m}_{K}(t,f)(x) := \int_{\R^{d}}K(h)|\Delta^{m}_{h}f(x)|\,dh, \quad\quad t > 0, x \in \R^{d}.
\]
In view of the domination $|K_{m}(t,f)| \leq d^{m}_{K}(t,f)$, this is certainly true for the inequality '$\lesssim$' in \eqref{eq:cor:difference_norms_separate_ineq_'sharp';BFS_equiv_norm_differences}. For the reverse inequality '$\lesssim$' one could try to extend the maximal function techniques from  \cite[Proposition~6]{S&S_jena-notes} to our setting via the square function variant of the Littlewood-Paley characterization~\eqref{eq:randomized-LP-decomp}; here one would have to replace the classical Hardy-Littlewood maximal function by the Banach lattice version from \cite{Bourgain_max,Garcia-Cuerva&Macias&Torrea_HL-property_Banach_Lattices,Rubio_de_Francia_max}.
\end{remark}

\begin{prop}\label{prop:R-bdd_convolutions_UMD_BFS}
Let $X$ be a UMD Banach function space, $p \in (1,\infty)$ and $w \in A_{p}(\R^{d})$.
Then $\mathscr{K}(\R^{d}) \hookrightarrow \mathcal{B}(L^{p}(\R^{d},w;X))$ maps bounded sets to $\mathcal{R}$-bounded sets.
\end{prop}
\begin{proof}
In the unweighted case $w=1$ this can be found in \cite[Section~4]{vNeerven&Veraar&Weis}. However, the Banach lattice version of the Hardy-Littlewood maximal operator is bounded on $L^{p}(\R^{d},w;X(\ell^{2}))$ for general $w \in A_{p}$, which  which is implicitly contained \cite{Garcia-Cuerva&Macias&Torrea_HL-property_Banach_Lattices}; also see~\cite{Tozoni_vector-valued_extensions}. Hence, the results from \cite[Section~4]{vNeerven&Veraar&Weis} remain valid for general $w \in A_{p}$.
\end{proof}

Recall that, given $k \in \mathscr{K}(\R^{d})$, for all $t > 0$ we have $k_{t} = t^{d}k(t\,\cdot\,) \in \mathscr{K}(\R^{d})$ with $\norm{k_{t}}_{\mathscr{K}(\R^{d})} = \norm{k}_{\mathscr{K}(\R^{d})}$. So, under the assumptions of the above proposition,
\[
\mathcal{R}\{ f \mapsto k_{t}*f : t > 0 \} \lesssim_{X,p,d,w} \norm{k}_{\mathscr{K}(\R^{d})} \quad \mbox{in} \quad \mathcal{B}(L^{p}(\R^{d},w;X)).
\]
In particular, if $m \in \Z_{\geq 1}$ and $K \in \mathscr{K}(\R^{d})$, then the choice $k=K^{\Delta^{m}}$ leads to the $\mathcal{R}$-boundedness of
$\{ f \mapsto K_{m}(t,f) : t > 0 \}$ in $\mathcal{B}(L^{p}(\R^{d},w;X))$.

\section{$1_{\R^{d}_{+}}$ as Pointwise Multiplier}\label{sec:pointwise-multiplier}

\subsection{Proof of Theorem~\ref{thm:pointwise_multiplier}}

Besides Theorem~\ref{thm:intro_main_result;difference} (or Theorem~\ref{thm:difference_norms_separate_ineq_'sharp'}), we need two lemmas for the proof of Theorem~\ref{thm:pointwise_multiplier}.
The first lemma says that the inclusion \eqref{eq:thm:pointwise_multiplier;inclusion} automatically implies its vector-valued version.
\begin{lemma}\label{lemma:inclusion_scalar->vector-valued}
Let $s \geq 0$, $p \in (1,\infty)$ and $w \in A_{p}(\R^{d})$. Let $w_{s,p}$ be the weight from Theorem~\ref{thm:pointwise_multiplier}.
If $H^{s}_{p}(\R^{d},w) \hookrightarrow L^{p}(\R^{d},w_{s,p})$, then there also is the inclusion
\begin{equation}\label{eq:thm:pointwise_multiplier;inclusion;vector-valued_version}
H^{s}_{p}(\R^{d},w;X) \hookrightarrow L^{p}(\R^{d},w_{s,p};X)
\end{equation}
for any Banach space $X$.
\end{lemma}
\begin{proof}
This can be shown as in \cite[Proof of Theorem~1.3,pg.~8]{Meyries&Veraar_char_class_embeddings}, which is based on the fact that the Bessel potential operator $\mathcal{J}_{-s}$ ($s \geq 0$) is positive as an operator from $L^{p}(\R^{d},w)$ to $H^{s}_{p}(\R^{d},w)$ (in the sense that $\mathcal{J}_{-s}f \geq 0$ whenever $f \geq 0$).
\end{proof}

The second lemma is very similar to Theorem~\ref{thm:difference_norms_separate_ineq_'sharp'}(ii) and may be thought of as an $\R^{d}_{+}$-version for the case $m=1$.

\begin{lemma}\label{lemma:part_difference_norms_separate_ineq;lemma_pt-multiplier}
Let $X$ be a UMD Banach space, $s \in (0,1)$, $p \in (1,\infty)$ and $w \in A_{p}(\R^{d})$.
Let $K \in \mathscr{K}(\R^{d}) \cap L^{1}(\R^{d},(1+|\,\cdot\,|)^{d+3})$. For each $f \in L^{p}(\R^{d},w;X)$ we define
\[
[f]^{\#}_{H^{s}_{p}(\R^{d}_{+},w;X)} = [f]^{(K)}_{H^{s}_{p}(\R^{d}_{+},w;X)} := \sup_{J \in \N}\normB{ \sum_{j=-J}^{J}\varepsilon_{j}2^{js}K_{\R^{d}_{+}}(2^{-j},f)\,}_{X_{p,w}(\R^{d}_{+})},
\]
where we use the notation
\[
K_{\R^{d}_{+}}(t,f)(x) := \int_{\{h_{1} \geq -x_{1}t^{-1}\}}K(h)\Delta_{th}f(x)\,dh, \quad\quad t>0, x \in \R^{d}_{+}.
\]
If $\{ f \mapsto \tilde{K}_{t}*f : t=2^{-j}, j \in \Z_{\geq 1} \} \subset \mathcal{B}(L^{p}(\R^{d},w;X))$ is $\mathcal{R}$-bounded, then we have the estimate
\[
[f]^{\#}_{H^{s}_{p}(\R^{d}_{+},w;X)} \lesssim \norm{f}_{H^{s}_{p}(\R^{d},w;X)}, \quad\quad f \in L^{p}(\R^{d},w;X).
\]
\end{lemma}
\begin{proof}
Note that, for each $t>0$, $f \mapsto K_{\R^{d}_{+}}(t,f)$ is a well-defined bounded linear operator on $L^{p}(\R^{d},w;X)$ of norm $\lesssim_{p,d,w} \norm{K}_{\mathscr{K}(\R^{d})}$. Using that $s>0$, for $f \in L^{p}(\R^{d},w;X)$ we can thus estimate
\[
\normB{\sum_{j=-J}^{J}\varepsilon_{j}2^{js}K_{\R^{d}_{+}}(2^{-j},f)\,}_{X_{p,w}(\R^{d}_{+})} \lesssim \norm{f}_{L^{p}(\R^{d},w;X)} +
\normB{\sum_{j=1}^{J}\varepsilon_{j}2^{js}K_{\R^{d}_{+}}(2^{-j},f)\,}_{X_{p,w}(\R^{d}_{+})}.
\]
Now fix $f \in L^{p}(\R^{d},w;X)$ and write $f_{n}:=S_{n}f$ for $n \in \N$ and $f_{n}:=0$ for $n \in \Z_{<0}$.
Then
\begin{equation}\label{eq:lemma:part_difference_norms_separate_ineq;lemma_pt-multiplier;(ii)_sum_n}
\normB{\sum_{j=1}^{J}\varepsilon_{j}2^{js}K_{\R^{d}_{+}}(2^{-j},f)\,}_{X_{p,w}(\R^{d}_{+})}
\leq \sum_{n\in\Z}\normB{\sum_{j=1}^{J}\varepsilon_{j}2^{js}K_{\R^{d}_{+}}(2^{-j},f_{n+j})}_{X_{p,w}}
\end{equation}

We first estimate the sum over $n \in \Z_{>0}$ in \eqref{eq:lemma:part_difference_norms_separate_ineq;lemma_pt-multiplier;(ii)_sum_n}.
Since
\[
K_{\R^{d}_{+}}(2^{-j},f_{n+j})(x)
= \tilde{K}_{2^{-j}}*(1_{\R^{d}_{+}}f)\,(x) + \left(\int_{\{h_{1} \geq -x_{1}2^{j}\}}K(h)\,dh \right)\,f_{n+j}(x),
\]
we can estimate
\begin{eqnarray*}
\normB{ \sum_{j=1}^{J}\varepsilon_{j}2^{js}K_{\R^{d}_{+}}(2^{-j},f_{n+j})\,}_{X_{p,w}}
&\leq& \normB{ \sum_{j=1}^{J}\varepsilon_{j}2^{js}\tilde{K}_{2^{-j}}*(1_{\R^{d}_{+}}f_{n+j}) }_{X_{p,w}} \\
&& \:+\:\: \normB{ x \mapsto \sum_{j=1}^{J}\varepsilon_{j}2^{js}\left(\int_{\{h_{1} \geq -x_{1}2^{j}\}}K(h)\,dh \right)\,f_{n+j}(x)}_{X_{p,w}}.
\end{eqnarray*}
For the first term we can use the assumed $\mathcal{R}$-boundedness of the involved convolution operators and for the second term we can use the contraction principle, to obtain
\begin{eqnarray*}
\normB{ \sum_{j=1}^{J}\varepsilon_{j}2^{js}K_{\R^{d}_{+}}(2^{-j},f_{n+j})\,}_{X_{p,w}}
&\lesssim& \normB{ \sum_{j=1}^{J}\varepsilon_{j}2^{js}1_{\R^{d}_{+}}f_{n+j} }_{L^{p}(\Omega;L^{p}(\R^{d},w;X))}
 + \normB{ \sum_{j=1}^{J}\varepsilon_{j}2^{js}f_{n+j}}_{X_{p,w}} \\
&\leq&  2\,2^{-ns} \norm{ \sum_{j=1}^{J}\varepsilon_{j}2^{(n+j)s}f_{n+j}}_{X_{p,w}} \\
&\lesssim& 2^{-ns}\norm{f}_{H^{s}_{p}(\R^{d},w;X)}.
\end{eqnarray*}
Since $s>0$, it follows that the sum over $n \in \Z_{>0}$ in \eqref{eq:lemma:part_difference_norms_separate_ineq;lemma_pt-multiplier;(ii)_sum_n} can be estimated from above by $C\norm{f}_{H^{s}_{p}(\R^{d},w;X)}$ for some constant $C$ independent of $f$ and $J$.

We next estimate the sum over $n \in \Z_{\leq 0}$ in \eqref{eq:lemma:part_difference_norms_separate_ineq;lemma_pt-multiplier;(ii)_sum_n}.
For each $n \leq 0$ we have
\begin{eqnarray*}
\normB{\sum_{j=1}^{J}\varepsilon_{j}2^{js}K_{\R^{d}_{+}}(2^{-j},f_{n+j})}_{X_{p,w}}
&=& \normB{ x \mapsto \sum_{j=1}^{J}\varepsilon_{j}2^{js}\int_{\R^{d}}1_{[-2^{-j}h_{1},\infty) }(x_{1})K(h)\Delta_{2^{-j}h}f_{n+j}(x)\,dh}_{X_{p,w}} \\
&\leq& \int_{\R^{d}}|K(h)|\,\normB{ x \mapsto \sum_{j=1}^{J}\varepsilon_{j}2^{js}1_{[-2^{-j}h_{1},\infty) }(x_{1})\Delta_{2^{-j}h}f_{n+j}(x)}_{X_{p,w}}\,dh \\
&\leq& \int_{\R^{d}}|K(h)|\,\normB{ x \mapsto \sum_{j=1}^{J}\varepsilon_{j}2^{js}\Delta_{2^{-j}h}f_{n+j}(x)}_{X_{p,w}}\,dh,
\end{eqnarray*}
where we used the contraction principle \eqref{eq:contraction_principle} in the last step. We can now proceed as in the proof of Theorem~\ref{thm:difference_norms_separate_ineq_'sharp'}(ii) to estimate the sum over $n \in \Z_{\leq 0}$ in \eqref{eq:lemma:part_difference_norms_separate_ineq;lemma_pt-multiplier;(ii)_sum_n} by $C\norm{f}_{H^{s}_{p}(\R^{d},w;X)}$ for some constant $C$ independent of $f$ and $J$.
\end{proof}

\begin{proof}[Proof of Theorem~\ref{thm:pointwise_multiplier}]
In view of Lemma~\ref{lemma:inclusion_scalar->vector-valued}, we need to show that $1_{\R^{d}_{+}}$ is a pointwise multiplier on $H^{s}_{p}(\R^{d},w;X)$ if and only if there is the continuous inclusion \eqref{eq:thm:pointwise_multiplier;inclusion;vector-valued_version}.
Defining $\bar{w}_{s,p}$ as the weight on $\R \times \R^{d-1}$ given by $\bar{w}_{s,p}(x_{1},x') := |x_{1}|^{-sp}w(x_{1},x')$,
the inclusion \eqref{eq:thm:pointwise_multiplier;inclusion;vector-valued_version} is equivalent to the inclusion
\begin{equation}\label{eq:thm:pointwise_multiplier;inclusion;vector-valued_version;modified_weight}
H^{s}_{p}(\R^{d},w;X) \hookrightarrow L^{p}(\R^{d},\bar{w}_{s,p};X)
\end{equation}
because $H^{s}_{p}(\R^{d},w;X) \hookrightarrow L^{p}(\R^{d},w;X)$.
So we must show that $1_{\R^{d}_{+}}$ is a pointwise multiplier on $H^{s}_{p}(\R^{d},w;X)$ if and only if there is the continuous inclusion \eqref{eq:thm:pointwise_multiplier;inclusion;vector-valued_version;modified_weight}.

\textbf{Step I.} \emph{Let $K \in \mathcal{S}(\R^{d})$ satisfy $\hat{K}(0) \neq 0$.
For a function $g$ on $\R^{d}$ we write $g^{\varrho}$ for the reflection in the hyperplane $\{0\} \times \R^{d-1}$, i.e. $g^{\varrho}(x):=g(-x)$.
Then $1_{\R^{d}_{+}}$ is a pointwise multiplier on $H^{s}_{p}(\R^{d},w;X)$ if and only if
\begin{equation}\label{eq:bewijs_pt-multiplier;StepI}
\normB{ x \mapsto \Big( \sum_{j \in \Z}\Big| 2^{js}\int_{\{h_{1} \leq -x_{1}2^{j}\}}k(h)\,dh \Big|^{2} \Big)^{1/2}\norm{f(x)}_{X} }_{L^{p}(\R^{d}_{+},v)} \lesssim \norm{f}_{H^{s}_{p}(\R^{d},v;X)}
\end{equation}
for $f \in L^{p}(\R^{d},v;X)$, $v \in \{w,w^{\varrho}\}$, $k \in \{K,K^{\varrho}\}$.}

\textbf{Step I.(a)} \emph{$1_{\R^{d}_{+}}$ is a pointwise multiplier on $H^{s}_{p}(\R^{d},w;X)$ if and only if
\begin{equation}\label{eq:bewijs_pt-multiplier;StepI.(a)}
[1_{\R^{d}_{\pm}}f]_{H^{s}_{p}(\R^{d}_{\pm},w;X)}  \lesssim \norm{f}_{H^{s}_{p}(\R^{d},w;X)}, \quad\quad f \in L^{p}(\R^{d},w;X),
\end{equation}
where
\[
[f]_{H^{s}_{p}(\R^{d}_{\pm},w;X)} := \sup_{J \in \N}\normB{\sum_{j=-J}^{J}\varepsilon_{j}2^{js}K_{1}(2^{-j},f)}_{L^{p}(\Omega;L^{p}(\R^{d}_{\pm},w;X))}.\footnote{Recall from Section~\ref{subsec:sec:difference_norms;notation} that $K_{1}(t,f)(x) = \int_{\R^{d}}K(h)\Delta_{th}f(x)\,dh$.}
\]
}
Since $[g]^{(1,K);\Z}_{H^{s}_{p}(\R^{d},w;X)} = \big([g]_{H^{s}_{p}(\R^{d}_{-},w;X)}^{p} + [g]_{H^{s}_{p}(\R^{d}_{+},w;X)}^{p}\big)^{1/p} \eqsim [g]_{H^{s}_{p}(\R^{d}_{-},w;X)} + [g]_{H^{s}_{p}(\R^{d}_{+},w;X)}$ for $g \in L^{p}(\R^{d},w;X)$, it follows from Theorem~\ref{thm:intro_main_result;difference}
(and Remark~\ref{rmk:thm:difference_norms_separate_ineq_'sharp';rmk_sum_integers}) that
\begin{equation}\label{eq:bewijs_pt-multiplier;StepI.(a);equiv_norms}
\norm{g}_{H^{s}_{p}(\R^{d},w;X)}
\eqsim \norm{g}_{L^{p}(\R^{d},w;X)} + [g]_{H^{s}_{p}(\R^{d}_{-},w;X)} + [g]_{H^{s}_{p}(\R^{d}_{+},w;X)}, \quad\quad g \in L^{p}(\R^{d},w;X).
\end{equation}

First we assume that \eqref{eq:bewijs_pt-multiplier;StepI.(a)} holds true. For all $f \in L^{p}(\R^{d},w;X)$ we can then estimate
\begin{eqnarray*}
\norm{1_{\R^{d}_{+}}f}_{H^{s}_{p}(\R^{d},w;X)}
&\stackrel{\eqref{eq:bewijs_pt-multiplier;StepI.(a);equiv_norms}}{\lesssim}&
\norm{1_{\R^{d}_{+}}f}_{L^{p}(\R^{d},w;X)} + [1_{\R^{d}_{+}}f]_{\R^{d}_{-}} + [1_{\R^{d}_{+}}f]_{\R^{d}_{+}} \\
&\leq& \norm{f}_{L^{p}(\R^{d},w;X)} + [f]_{_{\R^{d}_{-}}} + [1_{\R^{d}_{-}}f]_{\R^{d}_{-}} + [1_{\R^{d}_{+}}f]_{\R^{d}_{+}} \\
&\stackrel{\eqref{eq:bewijs_pt-multiplier;StepI.(a)},\eqref{eq:bewijs_pt-multiplier;StepI.(a);equiv_norms}}{\lesssim}& \norm{f}_{_{H^{s}_{p}(\R^{d},w;X)} }.
\end{eqnarray*}

Next we assume that $1_{\R^{d}_{+}}$ is a pointwise multiplier on $H^{s}_{p}(\R^{d},w;X)$. Then the inequality in \eqref{eq:bewijs_pt-multiplier;StepI.(a)} for $\R^{d}_{+}$ follows directly from \eqref{eq:bewijs_pt-multiplier;StepI.(a);equiv_norms}. Since $1_{\R^{d}_{-}} = 1 - 1_{\R^{d}_{+}}$, the inequality in \eqref{eq:bewijs_pt-multiplier;StepI.(a)} for $\R^{d}_{-}$ follows as well.

\textbf{Step I.(b)} \emph{\eqref{eq:bewijs_pt-multiplier;StepI} $\lra$ \eqref{eq:bewijs_pt-multiplier;StepI.(a)}.}
We only show that the inequality in \eqref{eq:bewijs_pt-multiplier;StepI.(a)} for $\R^{d}_{+}$ is equivalent to the inequality in \eqref{eq:bewijs_pt-multiplier;StepI} with $v=w$ and $k=K$, the equivalence of the other inequalities being completely similar.
We claim that the inequality in \eqref{eq:bewijs_pt-multiplier;StepI.(a)} for $\R^{d}_{+}$ is equivalent to the estimate
\begin{equation}\label{eq:bewijs_pt-multiplier;StepI(b);equiv_estimate}
\sup_{J \in \N}\normB{x \mapsto \sum_{j=-J}^{J}\varepsilon_{j}2^{js}\int_{h_{1} \leq -x_{1}2^{j}}K(h)\,dh\,f(x)}_{X_{p,w}(\R^{d}_{+})} \lesssim \norm{f}_{H^{s}_{p}(\R^{d},w;X)}, \quad\quad f \in L^{p}(\R^{d},w;X).
\end{equation}

Let us prove the claim. Note that, in view of the identity
\begin{equation*}\label{eq:bewijs_pt-multiplier;StepI.(b);simple_identity}
K_{1}(2^{-j},1_{\R^{d}_{+}}f)(x) = K_{\R^{d}_{+}}(2^{-j},f)(x)
+ \int_{\{h_{1} \leq -x_{1}2^{j}\}}K(h)\,dh\,f(x),
\end{equation*}
we have the inequalities
\begin{equation}\label{eq:bewijs_pt-multiplier;StepI.(b);simple_ineq1}
[1_{\R^{d}_{+}}f]_{H^{s}_{p}(\R^{d}_{+},w;X)} \leq [f]^{\#}_{H^{s}_{p}(\R^{d}_{+},w;X)} +
\sup_{J \in \N}\normB{x \mapsto \sum_{j=-J}^{J}\varepsilon_{j}2^{js}\int_{\{h_{1} \leq -x_{1}2^{j}\}}K(h)\,dh\,f(x)}_{X_{p,w}(\R^{d}_{+})}
\end{equation}
and
\begin{equation}\label{eq:bewijs_pt-multiplier;StepI.(b);simple_ineq2}
\sup_{J \in \N}\normB{x \mapsto \sum_{j=-J}^{J}\varepsilon_{j}2^{js}\int_{\{h_{1} \leq -x_{1}2^{j}\}}K(h)\,dh\,f(x)}_{X_{p,w}(\R^{d}_{+})}
\leq [1_{\R^{d}_{+}}f]_{H^{s}_{p}(\R^{d}_{+},w;X)} + [f]^{\#}_{H^{s}_{p}(\R^{d}_{+},w;X)}.
\end{equation}
Furthermore, note that the $\mathcal{R}$-boundedness condition from Lemma~\ref{lemma:part_difference_norms_separate_ineq;lemma_pt-multiplier} is fulfilled since $K \in \mathcal{S}(\R^{d})$; see Example~\ref{ex:thm:difference_norms_separate_ineq_'sharp';examples_cond_rho}.
Plugging the estimate from Lemma~\ref{lemma:part_difference_norms_separate_ineq;lemma_pt-multiplier} into \eqref{eq:bewijs_pt-multiplier;StepI.(b);simple_ineq1}, we see that \eqref{eq:bewijs_pt-multiplier;StepI(b);equiv_estimate} implies the inequality in \eqref{eq:bewijs_pt-multiplier;StepI.(a)} for $\R^{d}_{+}$. The reverse implication is obtained by
plugging the estimate from Lemma~\ref{lemma:part_difference_norms_separate_ineq;lemma_pt-multiplier} into \eqref{eq:bewijs_pt-multiplier;StepI.(b);simple_ineq2}.

Using the claim, this step is now completed by the observation that
\begin{eqnarray*}
&& \normB{x \mapsto \sum_{j=-J}^{J}\varepsilon_{j}2^{js}\int_{\{h_{1} \leq -x_{1}2^{j}\}}K(h)\,dh\,f(x)}_{X_{p,w}(\R^{d}_{+})} \\
&& \quad\quad\quad\quad = \: \normB{x \mapsto \normB{\sum_{j=-J}^{J}\varepsilon_{j}2^{js}\int_{\{h_{1} \leq -x_{1}2^{j}\}}K(h)\,dh}_{L^{p}(\Omega)}\norm{f(x)}_{X} }_{L^{p}(\R^{d}_{+},w)} \\
&& \quad\quad\quad\quad = \: \normB{ x \mapsto \Big( \sum_{j \in -J}^{J}\Big| 2^{js}\int_{\{h_{1} \leq -x_{1}2^{j}\}}K(h)\,dh \Big|^{2} \Big)^{1/2}\norm{f(x)}_{X} }_{L^{p}(\R^{d}_{+},w)}.
\end{eqnarray*}

\textbf{Step II.} \emph{Let $K = K^{[1]} \otimes K^{[2]} \in C^{\infty}_{c}(\R^{d})$, where $K^{[1]} \in C^{\infty}_{c}(\R)$ and $K^{[2]} \in  C^{\infty}_{c}(\R^{d-1})$ satisfy $K^{[1]} = K^{[1]}(-\,\cdot\,)$, $1_{[-1,1]} \leq K^{[1]} \leq 1_{[-2,2]}$ and $\widehat{K^{[2]}}(0)=1$. Then \eqref{eq:bewijs_pt-multiplier;StepI} is equivalent to \eqref{eq:thm:pointwise_multiplier;inclusion;vector-valued_version;modified_weight}.}
In view of the reflection symmetry $K=K^{\varrho}$, we only need to show that
\begin{equation}\label{eq:bewijs_pt-multiplier;StepII}
\Big( \sum_{j \in \Z}\Big| 2^{js}\int_{\{h_{1} \leq -y2^{j}\}}K(h)\,dh \Big|^{2} \Big)^{1/2} \eqsim y^{-s}, \quad\quad y \in \R_{+}.
\end{equation}
By the choice of $K$,
\[
\big| [-(1 \wedge y2^{j}),-y2^{j}] \big| \:\leq\: \int_{h_{1} \leq -y2^{j}}K(h)\,dh \:\leq\: \big| [-(2 \wedge y2^{j}),-y2^{j}] \big|, \quad\quad y \in \R_{+}.
\]
For every $b>0$ we have
\begin{eqnarray*}
\Big( \sum_{j \in \Z}\Big[ 2^{js}\big| [-(b \wedge y2^{j}),-y2^{j}] \big| \Big]^{2} \Big)^{1/2}
&\eqsim&  \Big( \int_{0}^{\infty} t^{-2s}\big| [-(b \wedge yt^{-1}),-yt^{-1}] \big|^{2} \frac{dt}{t} \Big)^{1/2} \\
&=& \Big( \int_{b^{-1}y}^{\infty} t^{-2s}(b-yt^{-1})^{2} \,\frac{dt}{t} \Big)^{1/2} \\
&=& b^{-(s+1)}y^{-s}\; \underbrace{\Big( \int_{1}^{\infty} \tau^{-2s-2}(\tau-1)^{2} \,\frac{d\tau}{\tau} \Big)^{1/2}}_{< \infty}.
\end{eqnarray*}
So we obtain \eqref{eq:bewijs_pt-multiplier;StepII} by taking $b=1,2$.
\end{proof}

\subsection{A Closer Look at the Inclusion $H^{s}_{p}(\R^{d},w) \hookrightarrow L^{p}(\R^{d},w_{s,p})$}\label{subsec:closer_look_inclusion}

In this section we give explicit conditions, in terms of $w$, $s$ and $p$, for which there is the continuous inclusion \eqref{eq:thm:pointwise_multiplier;inclusion} from Theorem~\ref{thm:pointwise_multiplier}. These conditions will be obtained from the following embedding result.

\begin{thm}(\cite[Theorem~1.2]{Meyries&Veraar_char_class_embeddings})\label{prop:Sobolev_emd_F-spaces_MV}
Let $w_{0},w_{1} \in A_{\infty}(\R^{d})$, $s_{0} > s_{1}$, $0 < p_{0} \leq p_{1} < \infty$, and $q_{0},q_{1} \in (0,\infty]$. Then there is the continuous inclusion
\begin{equation*}\label{eq:prop:Sobolev_emd_F-spaces_MV;inclusion}
F^{s_{0}}_{p_{0},q_{0}}(\R^{d},w_{0}) \hookrightarrow F^{s_{1}}_{p_{1},q_{1}}(\R^{d},w_{1})
\end{equation*}
if and only if
\begin{equation*}\label{eq:prop:Sobolev_emd_F-spaces_MV;C}
\sup_{\nu \in \N, m \in \Z^{d}}2^{-\nu(s_{0}-s_{1})}w_{0}(Q_{\nu,m})^{-1/p_{0}}w_{1}(Q_{\nu,m})^{1/p_{1}} < \infty,
\end{equation*}
where $Q_{\nu,m} = Q[2^{-\nu}m,2^{-\nu-1}] \subset \R^{d}$ denotes for $\nu \in \N$ and $m \in \Z^{d}$ the $d$-dimensional cube with sides parallel to the coordinate axes, centered at $2^{-\nu}m$ and with side length $2^{-\nu}$.
\end{thm}

\begin{prop}\label{prop:inclusie_expliciete_conditie}
Let $s>0$, $p \in (1,\infty)$ and $w \in A_{p}(\R^{d})$.
Suppose that $w_{s,p}(x_{1},x') = |x_{1}|^{-sp}w(x_{1},x')$ defines an $A_{\infty}$-weight on $\R^{d}=\R \times \R^{d-1}$. If
\begin{equation}\label{eq:prop:Sobolev_emd_F-spaces_MV;C;specific_situation;simplified}
\sup_{\nu \in \N,m \in \{0\}\times\Z^{d-1}}2^{-\nu sp}\frac{1}{w(Q_{\nu,m})}\int_{Q_{\nu,m}}|x_{1}|^{-sp}\,w(x)\,dx < \infty,
\end{equation}
then there is the continuous inclusion $H^{s}_{p}(\R^{d},w) \hookrightarrow L^{p}(\R^{d},w_{s,p})$.
In case that $w_{s,p} \in A_{p}$, the converse holds true as well.
\end{prop}
\begin{proof}
For the inclusion $H^{s}_{p}(\R^{d},w) \hookrightarrow L^{p}(\R^{d},w_{s,p})$ it is sufficient that $F^{s}_{p,2}(\R^{d},w) \hookrightarrow F^{0}_{p,1}(\R^{d},w_{s,p})$.
This follows from the identity $F^{s}_{p,2}(\R^{d},w) \stackrel{\eqref{eq:identity_Bessel-Potential_Triebel-Lizorkin;scalar-valued}}{=} H^{s}_{p}(\R^{d},w)$, denseness of $\mathcal{S}(\R^{d})$ in $H^{s}_{p}(\R^{d},w)$, the inclusion $(\mathcal{S}(\R^{d}),\norm{\,\cdot\,}_{F^{0}_{p,1}(\R^{d},w_{s,p})}) \stackrel{\eqref{eq:triebel_into_Lp;A_oneindig}}{\hookrightarrow} L^{p}(\R^{d},w_{s,p})$ and the fact that $H^{s}_{p}(\R^{d},w)$ and $L^{p}(\R^{d},w_{s,p})$ are both continuously included in the Hausdorff topological space $L^{0}(\R^{d})$. In the case that $w_{s,p} \in A_{p}$, there are the identities $F^{s}_{p,2}(\R^{d},w) = H^{s}_{p}(\R^{d},w)$ and $L^{p}(\R^{d},w_{s,p}) = F^{0}_{p,2}(\R^{d},w_{s,p})$ (see \eqref{eq:identity_Bessel-Potential_Triebel-Lizorkin;scalar-valued}), so the inclusion $H^{s}_{p}(\R^{d},w) \hookrightarrow L^{p}(\R^{d},w_{s,p})$ just becomes $F^{s}_{p,2}(\R^{d},w) \hookrightarrow F^{0}_{p,2}(\R^{d},w_{s,p})$.
Therefore, in order to prove the proposition, it is enough to show that, for every $q \in [1,\infty]$, the inclusion
\begin{equation}\label{eq:prop:inclusie_expliciete_conditie;incl_enough}
F^{s}_{p,2}(\R^{d},w) \hookrightarrow F^{0}_{p,q}(\R^{d},w_{s,p})
\end{equation}
is equivalent to the condition \eqref{eq:prop:Sobolev_emd_F-spaces_MV;C;specific_situation;simplified}.

By Theorem~\ref{prop:Sobolev_emd_F-spaces_MV}, the inclusion \eqref{eq:prop:inclusie_expliciete_conditie;incl_enough} holds true if and only if
\begin{equation*}\label{eq:prop:Sobolev_emd_F-spaces_MV;C;specific_situation}
\sup_{\nu \in \N, m \in \Z^{d}}2^{-\nu s}\norm{x \mapsto |x_{1}|^{-s}}_{L^{p}\left(Q_{\nu,m},\frac{1}{w(Q_{\nu,m})}w\right)} = \sup_{\nu \in \N, m \in \Z^{d}}2^{-\nu s}\left( \frac{w_{s,p}(Q_{\nu,m})}{w(Q_{\nu,m})} \right)^{1/p} < \infty.
\end{equation*}
But this condition is equivalent to \eqref{eq:prop:Sobolev_emd_F-spaces_MV;C;specific_situation;simplified}.
Indeed, for every $\nu \in \N$ and $m \in \Z^{d}$ with $m_{1} \neq 0$ we have
\[
|x_{1}| \geq (|m_{1}|-1/2)\,2^{-\nu} \geq \frac{1}{2}|m_{1}|\,2^{-\nu}, \quad\quad x \in Q_{\nu,m},
\]
implying that
\[
2^{-\nu s}\norm{x \mapsto |x_{1}|^{-s}}_{L^{p}\left(Q_{\nu,m},\frac{1}{w(Q_{\nu,m})}w\right)}
\leq 2^{s}|m_{1}|^{-s} \leq 2^{s}.
\]
\end{proof}

Let $d=n+k$ with $n,k \in \N$. For $\alpha,\beta > -n$ we define the weight $v_{\alpha,\beta}$ on $\R^{d} $ by
\begin{equation}\label{eq:example_power_weight}
v_{\alpha,\beta}(x,y) := \left\{\begin{array}{ll}
|x|^{\alpha} &\mbox{if}\:\: |x| \leq 1,\\
|x|^{\beta} &\mbox{if}\:\: |x| > 1,
\end{array}\right. \quad\quad (x,y) \in \R^{d} = \R^{n} \times \R^{k}.
\end{equation}
Given $p \in (1,\infty)$, we have $v_{\alpha,\beta} \in A_{p}$ if and only if $\alpha,\beta \in (-n,n(p-1))$; see \cite[Proposition~2.6]{Haroske&Skrzypczak_entropy_and_aprroximation_numbers}. For $n=1$ and $k=d-1$, we have $v_{\gamma,\gamma}=w_{\gamma}$ \eqref{eq:intro:power_weight} for every $\gamma > -1$.

\begin{ex}\label{ex:voorbeelden_inclusie}
Let $s >0$ and $p \in (1,\infty)$.
\begin{itemize}
\item[(i)] Suppose $w = w_{1} \otimes w_{2}$ with $w_{1} \in A_{p}(\R)$ and $w_{2} \in A_{p}(\R^{d-1})$. Then \eqref{eq:prop:Sobolev_emd_F-spaces_MV;C;specific_situation;simplified} reduces to the corresponding $1$-dimensional condition on $w_{1}$:
    \begin{equation}\label{ex:eq:prop:Sobolev_emd_F-spaces_MV;C;specific_situation;simplified;red_1-dim}
    \sup_{\nu \in \N}2^{-\nu sp}\frac{1}{w_{1}(Q_{\nu,0})}\int_{Q_{\nu,0}}|t|^{-sp}\,w_{1}(t)\,dt < \infty
    \end{equation}
\item[(ii)] Let $\alpha,\beta \in (-1,p-1)$. Consider the weight $w=v_{\alpha,\beta}$ from \eqref{eq:example_power_weight} for $n=1$ and $k=d-1$. There is the inclusion $H^{s}_{p}(\R^{d},w) \hookrightarrow L^{p}(\R^{d},w_{s,p})$ if and only if $s<\frac{1+\alpha}{p}$. Given a UMD space $X$, by Theorem~\ref{thm:pointwise_multiplier} we thus have that $1_{\R^{d}_{+}}$ is a pointwise multiplier on $H^{s}_{p}(\R^{d},v_{\alpha,\beta};X)$ if and only if $s<\frac{1+\alpha}{p}$. In the case $\alpha=\beta$ this is precisely \cite[Theorem~1.1]{Meyries&Veraar_pt-multiplication} restricted to positive smoothness; note that the general case $\alpha,\beta \in (-1,p-1)$ can be deduced from the case $\alpha=\beta \in (-1,p-1)$.
\end{itemize}
\end{ex}
\begin{proof}[Proof of (ii).] By (i) we may without loss of generality assume that $d=1$. Note that $w_{s,p}$ is the the weight $v_{\alpha-sp,\beta}$
\eqref{eq:example_power_weight} for $n=1$ and $k=0$.

First assume that there is the inclusion $H^{s}_{p}(\R^{d},w) \hookrightarrow L^{p}(\R^{d},w_{s,p})$. Since $C^{\infty}_{c}(\R^{d}) \subset H^{s}_{p}(\R^{d},w)$, it follows that $v_{\alpha-sp,\beta}=w_{s,p} \in L^{1}_{loc}(\R^{d})$. Hence, $\alpha-sp > -1$.

Conversely, assume that $s<\frac{1+\alpha}{p}$. Then $\alpha-sp \in (-1,p-1)$, so that $w_{s,p} = v_{\alpha-sp,\beta} \in A_{p}$.
Using that $s<\frac{1+\alpha}{p}$, a simple computation shows that \eqref{ex:eq:prop:Sobolev_emd_F-spaces_MV;C;specific_situation;simplified;red_1-dim} holds true for $w=v_{\alpha,\beta}$. By Proposition~\ref{prop:inclusie_expliciete_conditie} we thus obtain that $H^{s}_{p}(\R^{d},w) \hookrightarrow L^{p}(\R^{d},w_{s,p})$.
\end{proof}

\section*{Acknowledgements.}
The author is supported by the Vidi subsidy 639.032.427 of the Netherlands Organisation for Scientific Research (NWO).
The author would like to thank Mark Veraar for careful reading and useful comments and suggestions; in particular, for his suggestion to use the Mihlin-H\"older multiplier theorem \cite[Theorem~3.1]{Hytonen_Mihlin-Hormander}. He would also like to thank the anonymous referee for careful reading.

\bibliographystyle{plain}

\end{document}